\documentclass[12pt]{article}
\usepackage{amssymb,amsmath,amsthm}
\newtheorem{thm}{Theorem}[section]
\newtheorem{cor}[thm]{Corollary}

\newtheorem{lemma}[thm]{Lemma}
\newtheorem{prop}[thm]{Proposition}
\newtheorem{rem}[thm]{Remark}
\newtheorem{fact}[thm]{Fact}

\newtheorem{Def}[thm]{Definition} \numberwithin{equation}{subsection}
\def\R{{\mathbb R}}

\def\N{{\mathbb N}}

\def\P{{\mathbb P}}

\def\1{{\mathbf 1}}
\def\a0{{\aleph_0}}

\def\mi{\mu}

\def\eps{{\varepsilon}}

\newcommand{\Ra}{{\Rightarrow}}

\newcommand{\ra}{{\rightarrow}}

\newcommand{\Int}{{\rm Int}}

\newcommand{\supp}{{\rm supp}}
\newcommand{\cl}{{\rm cl}}
\newcommand{\spa}{{\rm span}}

\def\is#1#2{\left\langle #1 , #2 \right\rangle}

\newcommand{\E}{{\mathbb{E}}}

\newcommand{\cov}{{\rm Cov}}
\newcommand{\RO}{\R_{+}}
\newcommand{\KO}{K_{+}}

\newcommand{\vX}{{\ensuremath{\mathbf{X}}}}
\newcommand{\vY}{{\ensuremath{\mathbf{Y}}}}

\newcommand{\bA}{{\bar{A}}}
\newcommand{\bB}{{\bar{B}}}

\newcommand{\tK}{{\tilde{K}}}
\newcommand{\tS}{{\tilde{S}}}
\newcommand{\tU}{{\tilde{U}}}
\newcommand{\tD}{{\tilde{D}}}
\newcommand{\tT}{{\tilde{T}}}
\newcommand{\tKO}{\tilde{K}_{+}}
\newcommand{\teta}{{\tilde{\eta}}}

\newcommand{\bigtimes}{{\prod}}

\newcommand{\polarbuza}[1]{\renewcommand{\theenumi}{#1\arabic{enumi}}}

\newcounter{kilosledzia}
\newenvironment{pasztetowa}
  {\begin{enumerate}\setcounter{enumi}{\value{kilosledzia}}}
  {\setcounter{kilosledzia}{\value{enumi}}\end{enumerate}}

\author{Marcin Pilipczuk (malcin@duch.mimuw.edu.pl) \and
        Jakub Onufry Wojtaszczyk\thanks{Partially supported by MEiN Grant no 1 PO3A 012 29} (onufry@duch.mimuw.edu.pl)\\ 
Department of Mathematics, Computer Science and Mechanics \\ 
University of Warsaw\\ 
ul. Banacha 2, 02-097 Warsaw, Poland\\
}

\title{The negative association property for the absolute values of random variables equidistributed on a generalized Orlicz ball}

\begin{document}
\maketitle

\begin{abstract}
Random variables equidistributed on convex bodies have received quite a lot of attention in
the last few years. In this paper we prove the negative association property (which generalizes the subindependence of coordinate slabs) for generalized Orlicz balls. This allows us to give a strong concentration
property, along with a few moment comparison inequalities. Also, the theory of negatively associated variables
is being developed in its own right, which allows us to hope more results will be available.

Moreover, a simpler proof of a more general result for $\ell_p^n$ balls is given.
\end{abstract}

\tableofcontents
\section{Introduction}

\subsection{Notation}

We shall begin by introducing the notation used throughout the
paper. For any set $A$ by $\1_A$ we shall denote the characteristic function of $A$.
As usually, $\R$ and $\RO$ will denote the reals and 
the non-negative reals respectively. By
$\R^k$ we shall mean the $k$-dimensional Euclidean space equipped
with the standard scalar product $\is{\cdot}{\cdot}$, the
Lebesgue measure denoted by $\lambda$ or $\lambda_k$ and a system of orthonormal coordinates $x_1,x_2,\ldots,x_k$. By $\RO^k$ we mean the generalized
positive quadrant, that is the set $\{(x_1,\ldots,x_k) \in \R^k
: \forall_i\ x_i \geq 0\}$. For a given set $K\subset \R^k$ by $\KO$ we shall
denote the positive quadrant of $K$, that is $K \cap \RO^k$. For a
given set $A$ by $\bA$ we will denote the complement of $A$.

For a measure $\mi$ on $\R^n$ and an affine subspace $H \subset
\R^n$, by the {\em projection of $\mi$ onto $H$}  we mean the measure
$\mi_H$ defined by $\mi_H(C) = \mi(\{x \in \R^n : P(x) \in C\})$,
where $P$ is the orthogonal projection onto $H$. If $\mi$ is given by a density function $m$ and 
$K \subset H \subset \R^n$, then by the {\em restriction of $\mi$ to $K$} we mean the measure $\mi_{|K}$
on $H$ given with the density $m \cdot \1_K$. By the support of a function $m:X \to \R$, denoted $\supp m$, we
mean $\cl \{x \in X : m(x) \neq 0\}$. If $\mi$ is a measure, then by $\supp \mi$ we mean the smallest closed set $A$ such that $\mi(\bar{A}) = 0$. In the cases we consider, when $\mi$ will be given by a density $m$, we will always have $supp \mi = \supp m$.

We shall call a set $K \subset \R^n$ a {\em symmetric body} if it is
convex, bounded, central-symmetric (i.e. if $x \in K$
then $-x \in K$) and has a non-empty interior. A body $K \subset \R^n$ is called {\em
1-symmetric} if for any $(\eps_1,\ldots,\eps_n) \in \{-1,1\}^n$ and
any $(x_1,\ldots,x_n) \in K$ we have $(\eps_1 x_1, \ldots, \eps_n
x_n) \in K$. Such a body is sometimes called {\em unconditional}.

A function $f : \RO \ra \RO \cup \{\infty\}$ is called a {\em Young
function} if it is convex, $f(0) = 0$ and $\exists_x : f(x) \neq 0$, $\exists_{x\neq 0} : f(x) \neq \infty$. If we have $n$ Young
functions $f_1,\ldots,f_n$, then the set
$$K = \{(x_1,\ldots,x_n) : \sum_{i=1}^n f_i(|x_i|) \leq 1\}$$ is a 1-symmetric body in $\R^n$. Such a set
is called a {\em generalized Orlicz ball}, also known in the
literature as a modular sequence space ball.

We shall call a Young function $f$ {\em proper} if it does not attain the $+\infty$ value and $f(x) > 0$ for $x > 0$. 
A generalized Orlicz ball is called {\em proper} if it can be defined by proper Young functions.

If the coordinates of the space $\R^n$ are denoted $x_1,x_2,\ldots,x_n$, the appropriate Young functions
will be denoted $f_1,f_2,\ldots,f_n$, with the assumption $f_i$ is applied to $x_i$. If some of the coordinates
are denoted $x,y,z,\ldots$, the appropriate Young functions will be denoted $f_x, f_y, f_z, \ldots$, with the
assumption that $f_x$ is applied to $x$, $f_y$ to $y$ and so on.

A function $f: \R \ra \R$ is called increasing (decreasing) if $x \geq
y$ implies $f(x) \geq f(y)$ ($f(x) \leq f(y)$) --- we do not require
a sharp inequality. A function $f:\R^k \ra \R$ or $f:\RO^k \ra \R$
is called {\em coordinate-wise increasing (decreasing)}, if for $x_i \geq
y_i$, $i = 1,2,\ldots,n$ we have $f(x_1,\ldots,x_k) \geq
f(y_1,\ldots,y_n)$ ($f(x_1,\ldots,x_k) \leq f(y_1,\ldots,y_n)$). A
set $A \subset \RO^k$ is called a {\em c-set}, if for $x_i \geq y_i
\geq 0$, $i = 1,2,\ldots,n$ and $(x_1,\ldots,x_n) \in A$ we have
$(y_1,\ldots,y_n) \in A$. For a coordinate-wise increasing function
$f : \RO^k \ra \R$ the sets $f^{-1}((-\infty,t])$ are c-sets, and
conversely the characteristic function of a c-set is a
coordinate-wise decreasing function on $\RO^k$. Similarily a function $f : \RO^k \ra \R$ is {\em radius-wise increasing}
if $f(tx_1,tx_2,\ldots,tx_n) \geq f(x_1,x_2,\ldots,x_n)$ for $t > 1$, and a set $A$ is a 
{\em radius-set} if its characteristic function is radius-wise decreasing.

We say a function $f : \R^n \ra \R_+$ is {\em log-concave} if $\ln f$ is concave. A measure $\mi$ on $\R^n$ is called log-concave if for any nonempty $A, B \subset \R^n$ and $t \in (0,1)$ we have $\mi(tA + (1-t)B) \geq \mi(A)^t\mi(B)^{1-t}$. A classic theorem by Borell (see \cite{Bo}) states that any log-concave density not concentrated on any affine hyperplane has a density function, and that function is log-concave. A random vector in $\R^n$ is said to be log-concave if its distribution is log-concave.

A sequence of random variables $(X_1,\ldots,X_n)$ is said to be
{\em negatively associated}, if for any coordinate-wise increasing bounded
functions $f,g$ and disjoint sets $\{i_1,\ldots,i_k\}$ and
$\{j_1,\ldots,j_l\} \subset \{1,\ldots,n\}$ we have
\begin{equation} \label{NegAss} \cov \big(f(X_{i_1},\ldots,X_{i_k}), g(X_{j_1},\ldots,X_{j_l})\big) \leq 0.\end{equation}

We say that the sequence $(X_j)$ is {\em weakly negatively associated} if
inequality (\ref{NegAss}) holds for $l = 1$, and {\em very weakly
negatively associated} if (\ref{NegAss}) holds for $l = k = 1$.

For a 1-symmetric body $K \subset \R^n$ we can treat the body, or its positive
quadrant, as a probability space, with the normalized
Lebesgue measure as the probability. Formally, we consider $\Omega =
K$, the Borel subsets of $K$ as the $\sigma$-family and $\P =
\frac{1}{\lambda(K)}\lambda$ as the probability measure. We do similarly for $\KO$. We also
define $n$ random variables $X_1, \ldots, X_n$, with $X_i$ being the
$i$-th coordinate of a point $\omega \in \KO$ or $K$.

\subsection{Results}

Our main subject of interest is to prove negative associacion type properties 
for some classes of symmetric bodies in $\R^n$. An straightforward approach is bound to fail
due to the following proposition:

\begin{prop} If for a 1-symmetric body $K$ we consider the random vectors
uniformly distributed on $K$ (not just on $\KO$) and the coordinate variables are very weakly
negatively associated, then they are pairwise
independent, and thus $K$ is a rescaled cube.\end{prop}

\begin{proof} Take any $i, j \in \{1,\ldots,n\}$ and any increasing functions
$f, g : \R \ra \R$. Then $f^\circ(x) = -f(-x)$ is increasing too. $K$ is
1-symmetric, so $(X_i,X_j)$ has the same joint distribution as $(-X_i,X_j)$, so
$$\cov(f^\circ(X_i), g(X_j)) = \cov\big(-f(-X_i), g(X_j)) = -\cov(f(X_i),
g(X_j)\big).$$

If both $\cov(f(X_i),g(X_j))$ and $-\cov(f(X_i),g(X_j))$ are non-positive, then 
$\cov(f(X_i),g(X_j)) = 0$. This holds for every $i,j,f,g$. In particular for every $a,b$ we have
$$\P (X_i \in [a,\infty) \cap X_j \in [b,\infty)) - \P (X_i \in [a,\infty)) \cdot \P(X_j \in [b,\infty)) = \cov(\1_{[a,\infty)},\1_{[b,\infty)}) = 0.$$ A standard argument shows that $X_i$ and $X_j$ are independent, thus the density of $\1_K$ is a product density, so $K$ has to be a product of intervals.
\end{proof}

Thus, even very weak negative associacion for coordinate variables occurs only in the trivial case. The problem becomes more interesting if we look at the variables $|X_i|$ (or, equivalently, restrict ourselves to $X_i \geq 0$).

K. Ball and I. Perissinaki in \cite{BP} prove the subindependence of coordinate slabs for $\ell^p$ balls, from which very weak negative association of $(|X_1|,\ldots,|X_n|)$ is a simple consequence. In the paper \cite{ja} Corollary 3.2 states that  the sequence of variables $(|X_1|, \ldots, |X_n|)$ is very weakly negatively associated for generalized Orlicz balls.

In this paper we shall prove that for a generalized Orlicz ball the sequence of variables $(|X_1|, \ldots, |X_n|)$ is negatively associated:
\begin{thm}\label{orliczglownetw} Let $K$ be an generalized Orlicz ball, and let $X_i$ be the coordinates of a random vector uniformly distributed on $K$. Then the sequence $|X_i|$ is negatively associated.
\end{thm}
We shall also prove an even stronger property of $\ell_p^n$ balls:
\begin{thm}\label{lpglownetwierdzenie}
Take any $p \in [1,\infty)$ and any $n \in \N$. Let $m: \R_+ \ra \R_+$ be any log-concave function and let $\mi$ be the measure on $\R^n$ with the density at $x$ equal to $m(\|x\|_p^p)$ normalized to be a probability measure. Let $I = \{i_1,\ldots,i_k\}, J = \{j_1,\ldots,j_l\}$ be two disjoint subsets of $\{1,2,\ldots,n\}$, and let $f : \RO^{k} \ra \R$, $g : \RO^{l} \ra \R$ be any radius-wise increasing functions bounded on $\supp \mi$. Let $X = (X_1,X_2,\ldots,X_n)$ be the vector distributed according to $\mi$. Then $$\cov(f(|X_{i_1}|,|X_{i_2}|,\ldots,|X_{i_k}|),g(|X_{j_1}|,|X_{j_2}|,\ldots,|X_{j_l}|)) \leq 0.$$
\end{thm}
This is an equivalent of the above theorem, but the uniform distribution is replaced by the class of distribution with the density being a log-concave function of the $p$-th power of the $p$-th norm, and the coordinate-wise increasing function replaced by radius-wise decreasing functions.

Let us comment on the organization of the paper. In the following subsection we shall state the main results and show a few corollaries which motivate these results. Section 2 is a collection of general lemmas, which allow us to reformulate the problem in a simpler fashion. In Section 3 a simple proof for the $\ell_p^n$ result is given. Section 4 introduces the definitions used in dealing with the generalized Orlicz ball case and investigates the basic properties of the defined objects. Section 5 states the $\theta$-theorem, which is the main tool of the proof, and gives a part of the proof. Section 6 contains the second part of the proof, which is a large transfinite inductive construction. Finally Section 7 applies the $\theta$-theorem to obtain the result for generalized Orlicz balls.

\subsection{Motivations}
This study was motivated by a desire to link the results achieved in convex geometry in \cite{ABP} for $\ell_p$ balls and in
\cite{ja} for generalized Orlicz balls with an established theory, which will hopefully allow us to avoid repeating proofs already made in a more general case. For example, a form of the Central Limit Theorem for negative associated variables was already known in 1984 (see \cite{newman}). We also hope some new observations can be made using this approach.

The negative association property is stronger then the sub-independence of coordinate slabs, which has been studied in the context of the Central Limit Theorem (see \cite{ABP}, \cite{BP}). The statement of Theorem \ref{lpglownetwierdzenie} was motivated by Theorem 6 of \cite{4a}, where a proof of subindependence of coordinate slabs is given for a different class of measures with density dependent on the $p$-th norm, also including the uniform measure and the normalized cone measure on the surface.

An example that can prove useful for applications in convex geometry is a pair of comparison inequalities due to Shao (see \cite{shao}). First, notice that as $|X_i|$ are negatively associated, they remain negatively associated when multiplied by any non-negative scalars (which amounts to multiplying $X_i$ by any scalars) and after the addition of any constant scalars. Thus the vectors $|a_i X_i| - c_i$ are negatively associated for any $a_i, c_i \in \R$. Shao's inequalities, when applied to our case it will state the following:
\begin{thm} Let $K \subset \R^n$ be a generalized Orlicz ball, $(a_i)_{i=1}^n$ be any sequence of reals and $(X_i)_{i=1}^n$ be the coordinates of the random vector uniformly distributed on $K$. Then for any convex function $f: \R \ra \R$ we have 
$$\E f\Big(\sum_{i=1}^n |a_i X_i|\Big) \leq \E f\Big(\sum_{i=1}^n |a_i X_i^\star|\Big),$$ where $X_i^\star$ denote independent random variables with $X_i$ and $X_i^\star$ having the same distribution for each $i$. Additionally, if $f$ is increasing, then for any sequence of reals $(c_i)_{i=1}^n$ we have
$$\E f\Big(\max_{k = 1,2,\ldots,n} \sum_{i=1}^k |a_i X_i| - c_i \Big) \leq \E f\Big(\max_{k = 1,2,\ldots,n} \sum_{i=1}^k |a_i X_i^\star| - c_i \Big).$$\end{thm}

A more direct consequence is a moment comparision theorem suggested by R. Lata{\l}a (note we compare the moments of the sums of variables, and not their absolute values):
\begin{thm} Let $K \subset \R^n$ be a generalized Orlicz ball, $(a_i)_{i=1}^n$ be a sequence of reals and $(X_i)_{i=1}^n$ be the coordinates of the random vector uniformly distributed on $K$. Then for any even positive integer $p$ we have $$\E \Big(\sum_{i=1}^n a_i X_i\Big)^p \leq \E \Big(\sum_{i=1}^n a_i X_i^\star\Big)^p,$$ with $X_i^\star$ defined as before.\end{thm}

\begin{proof} When we open the brackets in $(\sum a_i X_i)^p$ the summands in which at least one $X_i$ appears with an odd exponent average out to zero, as $K$ is 1-symmetric. Thus what is left is a sum of elements of the form 
$$(a_i X_1)^{2\alpha_1} (a_2 X_2)^{2\alpha_2} \ldots (a_n X_n)^{2\alpha_n} = |a_1 X_1|^{2\alpha_1} |a_2 X_2|^{2\alpha_2} \ldots |a_n X_n|^{2\alpha_n}.$$
 If we put $f(a_1 x_1) = (a_1 x_1)^{2\alpha_1}$ and $g(a_2 x_2,\ldots,a_n x_n) = (a_2 x_2)^{2\alpha_2}\cdot \ldots\cdot (a_n x_n)^{2\alpha_n}$, applying negative association we get 
\begin{align*}
\E |a_1 X_1|^{2\alpha_1} |a_2 X_2|^{2\alpha_2} \ldots |a_n X_n|^{2\alpha_n} & \leq \E |a_1 X_1|^{2\alpha_1} \E |a_2 X_2|^{2\alpha_2} \ldots |a_n X_n|^{2\alpha_n}  \\ & = \E |a_1 X_1^\star|^{2\alpha_1} \E |a_2 X_2|^{2\alpha_2} \ldots |a_n X_n|^{2\alpha_n} \\ & =  \E |a_1 X_1^\star|^{2\alpha_1} |a_2 X_2|^{2\alpha_2} \ldots |a_n X_n|^{2\alpha_n} .\end{align*} Repeating this process inductively we separate all the variables and get \begin{align*} 
\E \Big(\sum_{i=1}^n a_i X_i\Big)^p & = \sum_{\alpha_1 + \ldots + \alpha_n = p \slash 2} C_{\alpha_1,\ldots,\alpha_n} \E |a_1 X_1|^{2\alpha_1} |a_2 X_2|^{2\alpha_2} \ldots |a_n X_n|^{2\alpha_n} \leq \\ &\leq \sum_{\alpha_1 + \ldots + \alpha_n = p \slash 2} C_{\alpha_1,\ldots,\alpha_n} \E |a_1 X_1^\star|^{2\alpha_1} |a_2 X_2^\star|^{2\alpha_2} \ldots |a_n X_n^\star|^{2\alpha_n} = \\ & = \E \Big(\sum_{i=1}^n a_i X_i^\star\Big)^p.\end{align*}\end{proof}

Finally, we can apply Shao's maximal inequality to get a exponential concentration of the euclidean norm. Theorem 3 in \cite{shao} states:
\begin{thm} Let $(X_i)_{i=1}^n$ be a sequence of negatively associated random variables with zero means and finite second moments. Let $S_k = \sum_{i=1}^k X_i$ and $B_n = \sum_{i=1}^n \E X_i^2$. Then for all $x>0$, $a > 0$ and $0 < \alpha < 1$
$$\P \Big(\max_{1 \leq k \leq n} |S_k| \geq x\Big) \leq 2\P(\max_{1\leq k \leq n} |X_k| > a) + \frac{2}{1-\alpha} \exp\Bigg(-\frac{x^2 \alpha}{2(ax + B_n)} \cdot \Big(1+\frac{2}{3}\ln\Big(1 + \frac{ax}{B_n}\Big)\Big)\Bigg).$$\end{thm}

We say $K \subset \R^n$ is {\em in isotropic position} if $\lambda_n(K) = 1$ and $\E X_i^2 = L_K^2$ for some constant $L_K$ (any bounded convex set with a non-empty interior can be moved into isotropic position by an affine transformation, for more on this subject see e.g. \cite{MS}). Notice that if $|X_i|$ are negatively associated and $f_i$ are increasing, then $f_i(|X_i|)$ are also negatively associated. Thus the sequence $(X_i^2 - L_K^2)_{i=1}^n$ for $X = (X_i)_{i=1}^n$ uniformly distributed on a generalized Orlicz ball is also negatively associated. The moments of log-concave variables are comparable (see for instance \cite{klo}, Section 2, remark 5), thus we have 
$$\E (X_i^2 - L_K^2)^2 = \E X_i^4 + L_K^4 - 2 L_K^2 \E X_i^2 = \E X_i^4 - L_K^4 \leq 5 L_K^4.$$
If we put $\alpha = 1\slash 2$ and $x = nt$ in Shao's inequality and apply the bound we got above for the variance we get 
\begin{cor}\label{shao1}Let $K \subset \R^n$ be a generalized Orlicz ball in isotropic position, and $(X_i)_{i=1}^n$ be the coordinates of the random vector uniformly distributed on $K$. Then for any $t > 0$, $a > 0$ we have:
\begin{align*} \P\Big(\max_{1 \leq k \leq n} \Big|\sum_{i=1}^k (X_i^2 - L_K^2)\Big| > n t\Big) & \leq 2\P \Big(\max_{1\leq k \leq n} |X_k^2 - L_K^2| > a\Big) + \\ & + 4 \exp\Bigg(-\frac{nt^2}{4(at + 5  L_K^4)} \cdot \bigg(1+\frac{2}{3}\ln\Big(1 + \frac{at}{5L_K^4}\Big)\bigg)\Bigg).\end{align*}\end{cor}
To apply this result probably an idea on what order of convergence is possible to achieve with this formula would be needed. To this end we give the following corollary:
\begin{cor}\label{shao2}Let $K \subset \R^n$ be a generalized Orlicz ball in isotropic position, and $(X_i)_{i=1}^n$ be the coordinates of the random vector uniformly distributed on $K$. Then for any $t > 0$ we have:
$$\P\Big(\Big|\frac{\sum_{i=1}^n X_i^2}{n} - L_K^2\Big| > t\Big) \leq C e^{-cnt^2} + C n e^{-c\sqrt[3]{nt}},$$
where $C$ and $c$ are universal constants independent of $t$, $n$ and $K$.\end{cor}

For $t > t_0$ a better bound (of the order of $e^{-t\sqrt{n}}$) is due to Bobkov and Nazarov (see \cite{BN}). However, frequently a bound for $t \ra 0$ is needed --- for instance the proof of the Central Limit Theorem for convex bodies uses bounds for the concentration of the second norm for small $t$ (see for instance \cite{ABP}). In full generality (ie. for an arbitrary log-concave isotropic measure and for arbitrary $t$) such a result is given in a very recent paper by Klartag (see \cite{nowyKlartag}) with worse exponents --- the bound for the probabilty is of the order of $e^{t^{ 3.33} n^{0.33}}$. Previous proofs of such results (see \cite{fgp}, \cite{staryKlartag}) gave a logarithmic dependence of the exponent on $n$. The bound given in the corollary above is very rough, and in any particular case it is very likely it may be improved. However, we give it in order to show an explicit exponential bound in the concentration inequality which is uniform for all generalized Orlicz balls in a given dimension and applies for any $t > 0$.

\begin{proof}
Obviously $$\P\Bigg(\Bigg|\frac{\sum_{i=1}^n X_i^2}{n} - L_K^2\Bigg| > t\Bigg) \leq \P\Bigg(\max_{1 \leq k \leq n} \Bigg|\sum_{i=1}^k (X_i^2 - L_K^2) \Bigg| > nt\Bigg),$$ so we have only to bound the right hand side in Corollary \ref{shao1}. Put $a = \sqrt[3]{n^2t^2}$. We know (see \cite{isotropic}) that $L_K^2$ is bounded by some universal constant $L$ independent of $n$ and $K$ for any 1-symmetric body in $\R^n$. If $c$ is small enough and $C$ large enough, then for $a < L_K^2$ we have $$C n e^{-c\sqrt[3]{nt}} = Cn e^{-c\sqrt{a}} \geq 1.$$ Thus we may consider only the case $a > L_K^2$.

In this case \begin{align*} \P(\max_{1\leq k \leq n} |X_k^2 - L_K^2| > a) & \leq n \max_{1 \leq k \leq n} \P(|X_k^2 - L_K^2| > a) =
n \max_k \P(X_k^2 > a + L_K^2) \\ & \leq n \max_k \P(X_k^2 > a) = n \max_k \P(|X_k| > \sqrt{a}).\end{align*} Due to the Brunn-Minkowski inequality $X_k$ is log-concave (see for instance \cite{ff}), we know that $Var(X_k) \leq L_K^2 < C$ and $\E X_k = 0$, and thus $P(|X_k| > t) \leq c_1 e^{-c_2t}$ for some universal constants $c_1$ and $c_2$ independent of the distribution of $X_k$ and of $t$ (Borell's Lemma, see for instance \cite{MS}). Thus we get
$$\P\Big(\max_{1\leq k \leq n} |X_k^2 - L_K^2| > a\Big) \leq c_1 e^{-c_2\sqrt{a}} = c_1 e^{-c_2 \sqrt[3]{nt}}.$$

In the second part we shall simply bound $$\Bigg(1+\frac{2}{3}\ln\Big(1 + \frac{at}{5L_K^4}\Big)\Bigg) \geq 1.$$ Then 
$$4 \exp\Bigg(-\frac{nt^2}{4(at + 5 L_K^4)} \cdot \bigg(1+\frac{2}{3}\ln\Big(1 + \frac{at}{5L_K^4}\Big)\bigg)\Bigg) 
\leq 4 \exp\big(-\frac{nt^2}{4n^{2\slash 3}t^{5\slash 3} + 20 L_K^4}\big) \leq C e^{-c\sqrt[3]{nt}} + C e^{-cnt^2}.$$
\end{proof}

\subsection{Acknowledgements}
We would very much like to thank Rafa{\l} Lata{\l}a, who encouraged us to write the paper, was the first person to 
read it and check the reasoning, and helped improve the paper in innumerable aspects. He also taught us most of what we know in the subject.

We would also like to thank prof. Stanis{\l}aw Kwapie{\'n}, who first suggested to us the idea of searching for negative-association type properties for convex bodies.

\section{Easy facts}

\subsection{Simplifying}

We want to prove inequality (\ref{NegAss}) for various classes of functions (coordinate-wise increasing in the case of Theorem \ref{orliczglownetw} and radius-wise increasing in the case of Theorem \ref{lpglownetwierdzenie}). We may assume $k+l = n$ by putting 
$\tilde{g}(x_{j_1},\ldots,x_{j_l},x_{r_1},\ldots,x_{r_{n-l-k}}) = g(x_{j_1},\ldots,x_{j_l})$. For convienience we shall assume that the
Lebesgue volume of $\KO$ is 1 (inequality
(\ref{NegAss}) is invariant under homothety).
It will be more convienient to work with c-sets or radius-sets than with functions, which motivates the following Lemma:

\begin{lemma}\label{prelim} Let $\mi$ be any probability measure on $\RO^n$ and let $X = (X_1,X_2,\ldots,X_n)$ be the random vector distributed according to $\mi$. Assume that for given $0 \leq k,l \leq n$ we have two families of bounded functions $\mathcal{F}$ on $\RO^k$ and $\mathcal{G}$ on $\RO^l$. Let $\mathcal{A} = \{f^{-1}(-\infty,t] : f \in \mathcal{F}, t\in \R\}$, and similarly $\mathcal{B}$ for $\mathcal{G}$. If for any $A \in \mathcal{A}$ and $B \in \mathcal{B}$ we have
\begin{equation}\label{wzorek_mi2}
\mi(A\times B) \mi(\bA \times \bB) \leq \mi(A \times \bB) \mi(\bA \times B),
\end{equation} 
then inequality (\ref{NegAss}) holds for $X$ and any $f \in \mathcal{F}, g \in \mathcal{G}$.
\end{lemma}

In particular, if inequality (\ref{wzorek_mi2}) holds for any $k$ and for any c-sets $A,B$, then the random variables $X_1,X_2,\ldots,X_n$ are negatively associated.

\begin{proof} Let us take any two functions $\mathcal{F} \ni f : \RO^k \ra \R$ and $\mathcal{G} \ni g: \RO^l \ra \R$.
As covariance is bilinear and is 0 if one of
the functions is constant, we may assume without loss of generality
that $f$ and $g$ are non-negative. For non-negative functions we
have $$f(x) = \int_0^\infty \1_{f^{-1}[t,\infty)} (x)\ dt.$$

Thus (again, by the bilinearity of the covariance) we can restrict
ourselves to functions $f$ and $g$ of the form $1 - \1_A$ and $1 -
\1_B$, where $A \in \mathcal{A}$ and $B \in \mathcal{B}$. Since $\cov(1-\1_A, 1-\1_B) = \cov(\1_A, \1_B)$, we have to
prove that $\cov(\1_A, \1_B) \leq 0$.

Let us denote by \vX \ the $k$-dimensional vector
$(X_{i_1},\ldots,X_{i_k})$ on which $f$ is taken, and by \vY \  the
$l$-dimensional vector on which $g$ is taken. Then
\begin{align*} \cov\big(\1_A(\vX),\1_B(\vY)\big) &= \E\1_A(\vX)\1_B(\vY) - \E\1_A(\vX) \E\1_B(\vY) = \mi(A \times B) - \mi(A \times \R^l) \mi(\R^k \times B) \\
&= \mi(A \times B) \mi\big((A \cup \bA) \times (B \cup \bB)\big) - \mi\big(A \times (B \cup \bB)\big) \mi \big((A \cup \bA) \times B\big) \\
&= \mi(A\times B) \mi(\bA \times \bB) - \mi (A \times \bB) \mi(\bA \times B),\end{align*}
which is non-positive by (\ref{wzorek_mi2}). \end{proof}

\subsection{Simple proportion lemmas}

During the course of further proofs we shall frequently need to
compare two ratios of integrals of the same functions over different sets.

In this subsection we will demonstrate some simple properties of ratios of integrals.

\begin{fact}\label{obvi} Let $a,b \geq 0$ and $c,d > 0$. Then the following are equivalent:
\begin{itemize}
\item $\frac{a}{c} \geq \frac{b}{d}$,
\item $\frac{a}{c} \geq \frac{a+b}{c+d}$,
\item $\frac{a+b}{c+d} \geq \frac{b}{d}$.
\end{itemize}
Whenever there is equality in one of the inequalities, all aforementioned fractions are equal.
\end{fact}

\begin{lemma}\label{rosncalk}\label{rosnmix}\label{gencheb}
Let $\mi$ be a non-negative measure on $\R$ supported on the (possibly unbounded) interval $[l_\mi,r_\mi]$. Suppose that $f, g, h:\R \ra \RO$ are functions bounded on $\supp \mi$, positive on the interior of their supports, satisfying:
\begin{enumerate}
\item The support of any function $u \in \{f,g,h\}$ is an interval $[l_u,r_u]$ (possibly unbounded),
\item $\frac{f}{g}$ is a decreasing function where defined, and $r_f \leq r_g$,
\item $h$ is an increasing function,
\end{enumerate}
Then:
\begin{enumerate}
\item[(1a)] For any $a<b<c$, $b \in (l_\mi, r_\mi) \cap (l_g,r_g)$ we have
$$\frac{\int_a^b f(x)d\mi}{\int_a^b g(x)d\mi} \geq \frac{f(b)}{g(b)} \hbox{ and } \frac{f(b)}{g(b)} \geq \frac{\int_b^c f(x)d\mi}{\int_b^c g(x)d\mi}$$ whenever both sides of an inequality are defined. 
\item[(1b)] Moreover, if for some $a < b < c$ we have two equalities in inequality (1a) then $\frac{f(x)}{g(x)}$ is constant on $(a,c) \cap \supp g \cap \supp \mi$ and for any $a \leq s < t \leq c$ $$\frac{\int_s^t f(x) d\mi}{\int_s^t g(x) d\mi}$$ is equal to $f(b) \slash g(b)$ if defined.
\item[(2a)] For any points $a,b,c,d$ satisfying $a < b \leq d$ and $a \leq c < d$ we have:
$$\frac{\int_a^b f(x)d\mi}{\int_a^b g(x)d\mi} \geq \frac{\int_c^d f(x)d\mi}{\int_c^d g(x)d\mi}$$ whenever both sides are defined.
\item[(2b)] Moreover, if this inequality is an equality and either $\int_a^c g(x) d\mi(x)$ or $\int_b^d g(x) d\mi(x)$ is strictly positive, then $\frac{f}{g}$ is constant on $[a,d]$ where defined, and we have an equality for any $a \leq a' \leq b' \leq d' \leq d$ and $c' \in [a',d']$ if both sides are defined.
\item[(3)] If $l_g = l_f$ the following inequality occurs for any interval $I$:
$$\frac{\int_I f(x) d\mi(x)}{\int_I g(x) d\mi(x)} \geq \frac{\int_I f(x) h(x) d\mi(x)}{\int_I g(x) h(x) d\mi(x)}$$ if both sides are defined.
\end{enumerate}
\end{lemma}

\begin{proof}
\begin{enumerate}
\item[(1a)]
Consider the first inequality. Let $a' = \max\{l_\mi,l_g,a\}$ . We have $a \leq a' < b$ (otherwise the denominator of the left-hand side would be undefined). Also $\int_a^b g(x) d\mi(x) = \int_{a'}^b g(x) d\mi(x) > 0$ and $g > 0$ on $(a',b]$ (it has to be positive in $b$ or the right-hand side would be undefined). Thus
$$ \frac{\int_a^b f(x)d\mi(x)}{\int_a^b g(x)d\mi(x)} \geq \frac{\int_{a'}^b f(x)}{\int_{a'}^b g(x)} = \frac{\int_{a'}^b g(x) \frac{f(x)}{g(x)}}{\int_{a'}^b g(x)} \geq
   \frac{\int_{a'}^b g(x) \frac{f(b)}{g(b)}}{\int_{a'}^b g(x)} = \frac{f(b)}{g(b)},$$
A similar reasoning with $c' = \min \{r_\mi,r_g,c\}$ proves the second inequality (note $r_f \leq r_g$, so the first inequality in the reasoning above becomes an equality).

\item[(1b)] If equality occurs, then $\frac{f(x)}{g(x)} = \frac{f(b)}{g(b)}$ for almost all $x \in (a',c')$ as $g$ is strictly positive on $(a',c')$. As $\frac{f}{g}$ is decreasing, if it is constant on almost whole $(a',c')$, it is constant on the whole interval and thus $$\frac{\int_s^t f(x) d\mi(x)}{\int_s^t g(s) d\mi(x)} = \frac{f(b)}{g(b)}$$ if defined for any $s,t \in (a',c')$. We know $\int_a^{a'} g(x) d\mi(x) = \int_{c'}^c g(x) d\mi(x) = 0$, so to have equalities we also have to have $\int_a^{a'} f(x) d\mi(x) = \int_{c'}^c f(x) d\mi(x) = 0$, thus $\int_s^t f(x)d\mi(x) = \int_{(s,t) \cap (a',c')} f(x) d\mi(x)$ and similarly for $g$, thus the thesis.

\item[(2a)] Let $F(x,y) = \int_x^y f(t)$ and $G(x,y) = \int_x^y g(t)$. As the left-hand side is defined, $G(a,b) > 0$ and thus $G(a,d) > 0$. We apply {\em(1a)} to get: 
\begin{equation}\label{insideeqrosnmix} \frac{F(a,b)}{G(a,b)} \geq \frac{F(b,d)}{G(b,d)}\end{equation}
if the right-hand side is defined and from Fact \ref{obvi} we have
$$\frac{F(a,b)}{G(a,b)} \geq \frac{F(a,b) + F(b,d)}{G(a,b) + G(b,d)} = \frac{F(a,d)}{G(a,d)}.$$
If the right-hand side in (\ref{insideeqrosnmix}) was not defined, $G(b,d) = 0$ and thus $F(b,d) = 0$ as $r_f \leq r_g$, so $\frac{F(a,b)}{G(a,b)} \geq \frac{F(a,d)}{G(a,d)}$. Similarly from {\em(1a)} $$\frac{F(a,c)}{G(a,c)} \geq \frac{F(c,d)}{G(c,d)}$$ if the left-hand side is defined, and thus from Fact \ref{obvi} $$\frac{F(a,d)}{G(a,d)} \geq \frac{F(c,d)}{G(c,d)}.$$ If the left-hand side was undefined, $G(a,d) = G(c,d)$ and obviously $F(a,d) \geq F(c,d)$, so we get the same inequality. Linking the two inequalities we get the thesis.

\item[(2b)] Suppose $G(b,d) > 0$. As $$\frac{F(a,b)}{G(a,b)} \geq \frac{F(a,d)}{G(a,d)} \geq \frac{F(c,d)}{G(c,d)}$$ and the first and last expressions are equal, all inequalities are in fact equalities. Thus from the first one of them and Fact \ref{obvi} we get $$\frac{F(a,b)}{G(a,b)} = \frac{F(b,d)}{G(b,d)},$$ and applying {\em(1b)} we get the thesis.

\item[(3)] Let $I' = I \cap \supp g$. As $\supp f \subset \supp g$ all integrals in the thesis over $I$ are equal to the appropriate integrals over $I'$. Consider the functions $h$ and $\frac{f}{g}$ on the interval $\Int I'$ (note $\frac{f}{g}$ is defined on $\Int I'$) taken with a measure with density $\frac{g(x)}{\int_{I'} g(t) d\mi(t)}d\mi$ (this is defined as the left-hand side in the thesis was defined, so $\int_{I'} g(t) d\mi(t) > 0$).
From the continuous Chebyshev sum inequality (that is, if $F$ is increasing and $G$ is decreasing, then $\int F\int G \geq \int FG \int 1$) we know
\begin{align*}\int_{I'} h(x) \frac{g(x)}{\int_{I'} g(t) d\mi(t)} d\mi(x) & \int_{I'} \frac{f(x)}{g(x)} \frac{g(x)}{\int_{I'} g(t) d\mi(t)} d\mi(x) \\ \geq & \int_{I'} h(x) \frac{f(x)}{g(x)} \frac{g(x)}{\int_{I'} g(t) d\mi(t)}d\mi(x) \int_{I'} \frac{g(x)}{\int_{I'} g(t) d\mi(t)}d\mi(x).\end{align*}
Multiplying both sides by $[\int_{I'} g(t) d\mi(t)]^2$ we get the thesis.
\end{enumerate}
\end{proof}

\begin{lemma} \label{rosdolp} Let $\mi$ be a non-negative measure on $I \subset \R$. Suppose $f,g,p,q: I \ra \RO$ are functions satisfying $f(x) g(y) \geq f(y) g(x)$ for $x \geq y$ and $p(x) q(y) \leq p(y) q(x)$ for $x \geq y$. Then 
$$\int_I p(x) f(x) d\mi(x) \int_I q(x) g(x) d\mi(x) \leq \int_I p(x) g(x) d\mi(x) \int_I q(x) f(x) d\mi(x).$$
\end{lemma} 

\begin{proof} Using Fubini's theorem we have to prove
$$\int_I \int_I p(x)f(x)q(y)g(y)\ d\mi(y)\ d\mi(x) \leq \int_I \int_I p(y)f(x)q(x)g(y)\ d\mi(y)\ d\mi(x).$$
Multiplying sides by two and changing names $x$ and $y$:
$$\int_I \int_I \big[p(x)f(x)q(y)g(y) + p(y)f(y)q(x)g(x) - p(x)f(y)q(y)g(x) - p(y)f(x)q(x)g(y)\big] \ d\mi(y)\ d\mi(x)\leq 0$$
$$\int_I \int_I \big(p(x)q(y) - p(y)q(x)\big)\big(f(x)g(y)-f(y)g(x)\big) \ d\mi(y) \ d\mi(x) \leq 0,$$
which follows from the assumptions, as the integrand is always non-positive.
\end{proof}

\begin{lemma}\label{dividesets}\label{dividepoints}
Suppose $f,g : X \ra \RO$ are defined on any set $X$ with a measure $\mi$. Let $\{ D_i\}_{i\in I}$ be a family of disjoint subsets
of $X$. If
$$t\int_{D_i} g(x) d\mi(x) \geq \int_{D_i} f(x) d\mi(x) \geq s\int_{D_i} g(x) d\mi(x)$$ for some $t,s \in \R \cup \{-\infty,\infty\}$, then
$$t\int_{\bigcup_i D_i} g(x) d\mi(x) \geq \int_{\bigcup_i D_i} f(x) d\mi(x)\geq s\int_{\bigcup_i D_i} g(x) d\mi(x).$$

If $X = X_1 \times X_2$ and $\mi = \mi_1 \otimes \mi_2$, and for some set $D\subset X_1 \times X_2$ and any $x_1 \in X_1$ we have
$$t\int_{(\{x_1\} \times X_2) \cap D} g(x) d\mi_2(x) \geq \int_{(\{x_1\} \times X_2) \cap D} f(x) d\mi_2(x) \geq s\int_{(\{x_1\} \times X_2) \cap D} g(x) d\mi_2(x),$$ then
$$t\int_D g(x) d\mi(x) \geq \int_D f(x) d\mi(x) \geq s\int_D g(x) d\mi(x).$$
\end{lemma}
\begin{proof}
In the first case, we should add all the inequalities by sides. In the second case, we should not sum but integrate using Fubini's theorem.
\end{proof}

\section{The $\ell_p^n$ ball case}

First we shall give the proof for $\ell_p^n$ balls. Recall the $\ell_p^n$ ball is the generalized Orlicz ball defined by the Young functions $f_i(x) = |x|^p$. We include this case for two reasons: first, it is much simpler than the Orlicz ball case, and serves as a good illustration of what is happening, and second, because we are able to achieve a stronger result, namely prove Theorem \ref{lpglownetwierdzenie}.

Note that in particular we can take $m$ to be $c_r \1_{[0,r]}$ to get the result for the uniform measure on the $\ell_p^n$ ball. As any coordinate-wise increasing function is radius-wise increasing, this result is stronger than the negative associacion property we prove for generalized Orlicz balls. By a simple approximation argument we can also get the result above for $\mi$ being the cone measure on the surface of $\ell_p^n$.

\begin{proof}
Let $B_p^n$ denote the $\ell_p^n$ ball. Let $M(x_1,x_2,\ldots,x_n) = (|x_1|,|x_2|,\ldots,|x_n|)$ and let $\tilde{\mi}$ be defined by $\tilde{\mi}(A) = \mi(M^{-1}(A))$. Notice $\tilde{\mi}$ describes the distribution of $(|X_1|,|X_2|,\ldots,|X_n|)$. As $\mi$ is 1-symmetric, we may equivalently define $\tilde{\mi}$ as $2^n$ times the restriction of $\mi$ to $\RO^n$.

Recall that the cone measure on $\partial B_p^n$ (that is, the boundary of $B_p^n$), which we shall denote $\nu$, is defined for $A \subset \partial B_p^n$ by 
$$\nu_n(A) = \frac{\lambda_n(ta : t \in \R, a \in A, ta \in B_p^n)}{\lambda_n(B_p^n)}.$$ For this measure we have the polar integration formula:
$$\int_{\R^n} f(x) dx = n \lambda_n(B_p^n) \int_{R_+} r^{n-1} \int_{\partial B_p^n} f(r\theta) d\nu_n(\theta) dr.$$
Let $C_n = n \lambda_n(B_p^n)$.

Due to Lemma \ref{prelim} we only need to prove inequality $\tilde{\mi}(A \times B) \tilde{\mi}(\bA \times \bB) \leq \tilde{\mi}(A\times \bB) \tilde{\mi} (\bA \times B)$ for any radius-sets $A,B$, which is equivalent to $\mi(A \times B)\mi(\bA\times \bB) \leq \mi(A\times \bB)\mi(\bA\times B)$. We have:
\begin{align*}
\mi(A\times B) &= \int_{\R^k} \int_{\R^{n-k}} \1_{A}(x) \1_B(y) m(\|x\|_p^p + \|y\|_p^p) dx dy = \\
&= \int_{\R_+} \int_{\partial B_p^k} \int_{\R^{n-k}} C_k r^{k-1} \1_{A}(r\theta) \1_B(y) m(r^p + \|y\|_p^p) d\nu_k(\theta) dr dy \\
&= \int_{\R_+} \bigg[\int_{\R^{n-k}} \1_B(y) m(r^p + \|y\|_p^p) dy \bigg] \bigg[\int_{\partial B_p^k} \1_A(r\theta) d\nu_k(\theta)\bigg] C_k r^{k-1} dr.\end{align*}

Denote $f_B(r) = \int_{\R^{n-k}} \1_B(y) m(r^p + \|y\|_p^p) dy$ and $g_A(r) = \int_{\partial B_p^k} \1_A(r\theta) d\nu_k(\theta)$. Let $\sigma_1$ be the measure on $\R_+$ with density $C_k r^{k-1}$. We can perform similar operations for the other three expressions in inequality (\ref{wzorek_mi2}). What we have to prove becomes the inequality
$$\int_{\R_+} f_B(r) g_A(r) d\sigma_1(r) \int_{\R_+} f_{\bB}(r) g_{\bA}(r) d\sigma_1(r) \leq \int_{\R_+} f_{\bB}(r) g_A(r) d\sigma_1(r) \int_{\R_+} f_B(r) g_{\bA}(r) d\sigma_1(r).$$ Due to lemma \ref{rosdolp} it is enough to prove the following two inequalities: \begin{eqnarray}
f_B(r_1) f_{\bB}(r_2) \geq f_B(r_2) f_{\bB}(r_1) \hbox{ for } r_1 \geq r_2, \label{fffdd} \\
g_A(r_1) g_{\bA}(r_2) \leq g_A(r_2) g_{\bA}(r_1) \hbox{ for } r_1 \geq r_2. \label{dddff}
\end{eqnarray}

Inequality (\ref{dddff}) is simple --- $\1_A(r\theta)$ is decreasing as a function of $r$ for any fixed $\theta$, while $\1_{\bA}(r\theta)$ is increasing, as $A$ is a radius-set. Thus $g_A(r)$ is decreasing, $g_{\bA}$ is increasing, so $g_A(r_1) \leq g_A(r_2)$ and $g_{\bA}(r_2) \leq g_{\bA}(r_1)$.

Inequality (\ref{fffdd}) will require a bit more work. We have:
\begin{align*}
f_B(r_1) &= \int_{\R^{n-k}} \1_B(y) m(r_1^p + \|y\|_p^p) dy \\
&= \int_{\R_+} \int_{\partial B_p^{n-k}} C_{n-k} s^{n-k-1} \1_B(s\xi) m(r_1^p + s^p) d\nu_{n-k}(\xi) dr \\
&= \int_{\R_+} \bigg[m(r_1^p + s^p)\bigg] \bigg[\int_{\partial B_p^{n-k}} \1_B(s\xi) d\nu_{n-k}(\xi)\bigg] C_{n-k} s^{n-k-1} ds.\end{align*}

We are going to use Lemma \ref{rosdolp} once again. Let $p_{r_1}(s) = m(r_1^p + s^p)$ and $q_B(s) = \int_{\partial B_p^{n-k}} \1_B(s\xi) d\nu_{n-k}(\xi)$ and $\sigma_2$ the measure with density $C_{n-k} s^{n-k-1}$. We do the similar calculation for the other three expressions in inequality (\ref{fffdd}), and it becomes
$$\int_{\R_+} p_{r_1}(s) q_B(s) d\sigma_2(s) \int_{\R_+} p_{r_2}(s) q_{\bB}(s) d\sigma_2(s) \geq \int_{\R_+} p_{r_2}(s) q_B(s) d\sigma_2(s) \int_{\R_+} p_{r_1}(s) q_{\bB}(s) d\sigma_2(s).$$
Applying Lemma \ref{rosdolp} we have to prove 
\begin{eqnarray}
p_{r_1}(s_1) p_{r_2}(s_2) \leq p_{r_2}(s_1) p_{r_1}(s_2) \hbox{ for } s_1 \geq s_2, \label{fdfd} \\
q_B(s_1) q_{\bB}(s_2) \leq q_{\bB}(s_1) q_B(s_2) \hbox{ for } s_1 \geq s_2. \label{dfdf}
\end{eqnarray}

Inequality (\ref{dfdf}) is proved in the same way as inequality (\ref{dddff}) --- $q_B$ is decreasing and $q_{\bB}$ is increasing. Inequality (\ref{fdfd}) means $$m(r_1^p + s_1^p) m(r_2^p + s_2^p) \leq m(r_2^p + s_1^p) m(r_1^p + s_2^p),$$ which follows from the log-concavity of $m$.
\end{proof}

As we saw, this proof was quite simple. Unfortunately, it takes advantage of the fact that the Young function of the $\ell_p^n$ ball scales well with the radius, that is, that $f_i(tx_i) = \phi(t)f_i(x_i)$ for some function $\phi$. Of all Orlicz ball only the $\ell_p$ balls have this property, which makes it impossible to apply the same proof to the generalized Orlicz ball case.

\section{The generalized Orlicz ball case --- preliminaries, the proper measure, lens sets}

\subsection{Idea of the proof}

We would like to transfer the result given above for $\ell_p^n$ balls to the more general case of generalized Orlicz balls. In the generalized Orlicz ball cas the Young function does not, unfortunately, scale with the radius, and this creates the need for a different approach. Again by Lemma \ref{prelim} we can restrict ourselves to characteristic functions of c-sets. As generalized Orlicz balls are 1-symmetric, we can restrict ourselves to the positive quadrant of our generalized Orlicz ball.

We shall proceed in two steps. The first will be to prove that generalized Orlicz balls satisfy inequality (\ref{NegAss}) if one of the functions, say $g$, is univariate --- in other words, to begin by proving weak negative association. This is equivalent to proving \ref{wzorek_mi2} for one of the sets, say $B$, being one-dimensional. Due to Lemma \ref{rosncalk}, part 1, we will simply need to prove that the function $\frac{\lambda_{n-1}(A \times \{z\} \cap K)}{\lambda_{n-1}(\bA \times \{z\} \cap K)}$ is decreasing with $z$.
Thus, we take any $z_2 > z_1 \geq 0$ and concentrate on them.

We want to prove $$\frac{\lambda_{n-1}(A \times \{z_1\} \cap K)}{\lambda_{n-1}(\bA \times \{z_1\} \cap K)} \leq \frac{\lambda_{n-1}(A \times \{z_2\} \cap K)}{\lambda_{n-1}(\bA \times \{z_2\} \cap K)}.$$ 
Switching the right denominator with the left numerator we get
$$\frac{\lambda_{n-1} ((\{z_2\} \times A) \cap K)}{\lambda_{n-1} ((\{z_1\} \times A) \cap K)} \geq \frac{\lambda_{n-1} ((\{z_2\} \times \bA) \cap K)}{\lambda_{n-1} ((\{z_1\} \times \bA) \cap K)}$$
as the inequality we need to prove. We shall denote the proportion of the measure of $K_{z_2}$ to the measure of $K_{z_1}$ on a given set $D$ by $\theta(D)$. 

The second step will be to pass from the univariate case to the general case. It turns out that a very similar argument, using the proportion $\frac{\lambda(D\cap \bB)}{\lambda(D \cap B)}$ as $\theta(D)$ will allow us to do that. Thus, to avoid repetition (as the argument is quite long), we shall take the properties of both of these functions which make the similar arguments possible and call any function with such properties a $\Theta$-function, then attempt to prove \begin{equation}\label{ttt}\theta(K \cap A) \geq \theta(K) \geq \theta(K \cap \bA)\end{equation} for any $\Theta$-function $\theta$.

Section 4 is devoted to defining the concepts used in the proof (subsection 4.2) and proving general lemmas about those concepts (subsections 4.3, 4.4 and 4.5). In particular, the properties defining a $\Theta$-function are given. Section 7 assumes inequality \ref{ttt} and proves Theorem \ref{orliczglownetw}. Sections 5 and 6 are devoted to the proof of inequality \ref{ttt}.

The idea of Section 7 is quite simple --- a Brunn-Minkowski argument and a few approximations are enough to verify that the appropriate functions considered for generalized Orlicz balls are in fact $\Theta$-functions. The main line of the reasoning is similar to \cite{ja}.

To prove inequality \ref{ttt} we shall attempt to divide the set $\KO$ into appropriately small convex subsets $D$ for which $\theta(D) = \theta(K)$. On each of these sets we will prove inequality (\ref{ttt}) with $D$ substituted for $K$, which proves the thesis ($\theta$ is a proportion, so if it is attains some value on a family of disjoint sets, it attains the same value on the sum of this family). The problem, of course, is to prove the inequality (\ref{ttt}) for any set $D$ (this is the aim of Section 5) and to construct a division into suitable sets $D$ (this is the aim of Section 6).

For Section 5, the sets $D$ will have to be of the form $\tilde{D} \times \R^{n-2}$, where $\tilde{D}$ is 2-dimensional. Moreover, we will need $\tilde{D}$ to be ``long and narrow''. This will allow us to take one direction
(the one in which $\tilde{D}$ is ``long'') to be a new coordinate, replacing the two coordinates of $\tilde{D}$, and to approximate the set $A$ and the function $\theta$ on $D$ with their approximations constant in the other, ``narrow'', variable. If the approximation is good enough (and it turns out to be), we can inductively use the inequality (\ref{ttt}) for the $n-1$ dimensional case for the approximating functions and then transfer the result to the original functions.

We cannot reasonably expect the sets $D$ to have constant width in the ``narrow'' coordinate. This means that in the inductive step we shall have to consider weighted measures to take this into account. This motivates us to consider a more general theorem, in which the Lebesgue measure on $K$ will be replaced by a proper weighted measure.

The argument in Section 6 is somewhat similar to the Kanaan--Lovasz--Simonovits localization lemma. However, we need the sets $D$ to satisfy additional assumptions, in particular to be ``positively inclined'' (this roughly means that the ``long'' coordinate axis has to be of the form $y = ax + b$, where $a$ is positive). We were unable to fit this into the localization lemma scheme, so the division is done by hand.

We prove in Section 5 we can cut off a ``good'' set $D$ from our ball. Unfortunately, we have no control of the measure of the set we cut off (apart from the fact it is positive). Thus inductive cutting off good sets does not necassarily cover the whole $K$. This leads us to a transfinite inductive reasoning, where we cut off ``good'' sets in a transfinite fashion (that is, after cutting off countably many we see what is left and continue cutting). This approach leads to a number of technical problems associated with the limit step, and Section 6 is devoted to dealing with these problems and following through with the transfinite induction.

\subsection{Definitions}
For the convienience of the reader all the basic definitions have been gathered in one place. So here we will just introduce the concepts required in the proof, and the next sections will be devoted to gaining a deeper understanding of those concepts.

We shall usually consider a generalized Orlicz ball $K \subset \R_x \times \R_y \times \R^{n-2}$.
By $K_{x = u}$ we shall mean the section of $\KO$ with the hyperplane $x = u$, similarly for any other variable in $\R^n$.

For a given set $D \subset \R_x \times \R_y \times \R^{n-2}$ by $\tD$ we shall denote the projection of $D$ to $\R_x \times \R_y$.
If not said otherwise, we shall assume $D = \tD \times \R^{n-2}$.

\begin{Def} A function $f:\R^n \ra [0,\infty)$ is called {\em $1\slash m$-concave} if its support is a convex set and the function $f^{1\slash m}$ is concave on its support.\end{Def}

\begin{Def} \label{propermeasure} Let $K \subset \R^n$ be a generalized Orlicz ball.
A measure $\mi$ on $\R^n$ is called a {\em proper measure} with respect to $K$ for $\R^n = \R_x \times \R_y \times \R^{n-2}$ ($n \geq 2$) if the following conditions are satisfied:
\begin{itemize}
\item $\mi$ is a non-negative measure with density $f(x) g(y) \1_{\KO}$.
\item The functions $f$ and $g$ are $1\slash m$-concave for some $m > 0$.
\item If $K_{x = x_0} = \emptyset$ for a given $x_0$ then $f(x_0) = 0$, and if $K_{y = y_0} = \emptyset$ for a given $y_0$ then $g(y_0) = 0$.
\end{itemize}
In the case $n = 1$ a proper measure is a non-negative measure with a $1\slash m$-concave density $f$ for some $m > 0$, satisfying $\supp f \subset \KO$.
\end{Def}

This definition describes the ``proper weighted measures'' which we will have to analyze in the subsequent induction steps of the proof outlined above.

We shall denote the support of $f$ by $[x_-,x_+]$ and the support of $g$ by $[y_-,y_+]$.
Of course $0 \leq x_- \leq x_+$ and similarly for $y$.

If we have a proper measure on $\R^n$ with respect to $K$ we can define a {\em lens set}. This definition describes the shape of a set, which will be one of the conditions of ``not losing too much on approximation'' and also will be a condition under which further dividing will be possible.

\begin{Def}
A set $D \subset \R_x \times \R_y \times \R^{n-2}$ is called a {\em lens set} if:
\begin{itemize}
\item $D$ is a convex set,
\item $D = \tilde{D} \times \R^{n-2}$,
\item for some $x_- \leq  x_1 < x_2 \leq x_+$ and $y_- \leq y_1 < y_2 \leq y_+$, we have $\tD \subset [x_1,x_2] \times [y_1, y_2]$ and $(x_1, y_1) \in \tD$ and $(x_2, y_2) \in \tD$,
\item $\mi(D) > 0$.
\end{itemize}
A lens set is said to be a {\em strict lens set} if %!!! DODANE
$x_- < x_1 < x_2 < x_+$, $y_- < y_1 < y_2 < y_+$ and %!!! DOTAD
points $(x_1, y_1)$ and $(x_2, y_2)$ are the only points of $\cl \tD$ belogning to the boundary of the rectangle $[x_1, x_2] \times [y_1, y_2]$.
\end{Def}

Note that the boundary of the projection of a strict lens set onto $\R_x \times \R_y$ consists of an upper part, which is a graph of an concave, strictly increasing function, and a lower part, which is the graph of a convex, strictly increasing function. The boundary of a (non-strict) lens set may additionaly contain horizontal and vertical intervals adjacent to $(x_1,y_1)$ and $(x_2,y_2)$. We shall speak of the upper-left border and the lower-right border of a lens set. 

For a lens set $D$ we define the {\em extremal points} of $\tD$ to be two points $(x_1,y_1)$ and $(x_2,y_2)$.
From the definition of a lens set, the extremal points belong to $\tD$. The {\em extremal line} of a lens set
is the line connecting extremal points. By the {\em width} of a lens set we shall mean the length of its projection upon the
line perpendicular to its extremal line in the plane $\R_x \times \R_y$.

\begin{Def} For a line $L$ in $\R_x \times \R_y$ the {\em inclination} of $L$ will denote measure of the angle between $\R_x$ and $L$ oriented so that the inclination of the line $\{x = y\}$ is $\pi \slash 4$. A line is said to have {\em positive inclination} if its inclination belongs to $(0,\pi \slash 2)$, and {\em non-negative inclination} if the inclination belongs to $[0,\pi\slash 2]$.
The inclination of a lens set $D$ is simply the inclination of its extremal line.

By a {\em positively inclined hyperplane} in $\R^n$ we mean a hyperplane $H$ defined by $x_i = \lambda x_j + c$, where $\lambda \geq 0$.
\end{Def}

\begin{Def} For a given convex set $D$ and a proper measure $\mi$ by the relevant diameter of $D$ we mean the diameter of $D \cap \supp \mi$.\end{Def}
% No i gdzie/s b/edzie trzeba doda/c (zapewne przy tych lematach dlugi-waski) fakt, ze jak zbior jest waski i ma duzy relevant diameter, to musi miec duze przeciecie $\mi \cap D \cap $extremal line. !!! CZY TO JEST POTRZEBNE?

\begin{Def} For a given generalized Orlicz ball $K \subset \R^n$ by its {\em restriction to a positively inclined hyperplane $H$} we mean such a generalized Orlicz ball $K' \subset \R^{n-1}$ such that $\KO \cap H$ is isometric to $\KO'$. By Lemma \ref{sekcjajest} there exists such a generalized Orlicz ball $K'$.\end{Def}

\begin{Def} For a given generalized Orlicz ball $K \subset \R^n$ by its {\em restriction to an interval $I \subset \RO$ with respect to the coordinate $x_i$} we mean such a generalized Orlicz ball $K' \subset \R^n$ that $\KO'$ is isometric to $K \cap \{x_i \in I\}$. By Lemma \ref{obciacboki} there exists such a generalized Orlicz ball $K'$. When it is obvious in which coordinate the interval $I$ is taken we shall simply write that $K'$ it a restriction of $K$ to $I$.\end{Def}

\begin{Def} For a given generalized Orlicz ball $K \subset \R^n$ and a generalized Orlicz ball $K' \subset \R^m$ we say that $K'$ is a {\em derivative} of $K$ if there exists a sequence $K = K_0, K_1,\ldots,K_n = K'$ of generalized Orlicz balls such that for each $i \in \{1,\ldots,n\}$ the ball $K_i$ is either a restriction of $K_{i-1}$ to some positively inclined hyperplane or a restriction of $K_{i-1}$ with respect to some variable $x_k$ to some interval $I \subset \RO$.\end{Def}

We can embed isometrically the positive quadrant of any derivative of $K$ into the positive quadrant of $K$. We shall identify without notice the positive quadrant of the derivative with the image of this embedding in the positive quadrant of $K$. In particular for a function $f$ defined on $\KO$ we shall speak of its restriction to $\KO'$, meaning such a function $\tilde{f}$ that $\tilde{f}(x) = f(\phi(x))$, where $\phi$ is the embedding of $\KO'$ into $\KO$.

For the space $\R^n$ with a fixed orthonormal system $e_1,\ldots,e_n$ by a {\em coordinate-wise decompostion} of $\R^n$ we mean a decompostion $\R^n = \R^k \times \R^l$, where $\R^k = \spa\{e_{i_1},\ldots,e_{i_k}\}$ and $\R^l = \spa\{e_{j_1},\ldots,e_{j_l}\}$, with $i_p \neq j_q$ for any $p,q$.

The main tool used in this proof will be the $\Theta$ functions. We define the $\Theta$ functions as follows: 

\begin{Def} For a given generalized Orlicz ball $K \subset \R^n$ and two functions $\eta_1, \eta_2$ defined on $\KO$ we say that $\eta_1$ and $\eta_2$ define a $\Theta$ function on $K$ if the following properties are satisfied:

\begin{pasztetowa}\polarbuza{T}
\item\label{t-1} The functions $\eta_1$ and $\eta_2$ are bounded.
\item\label{t0} The functions $\eta_1$ and $\eta_2$ are coordinate-wise non-increasing.
\item\label{t1} We have $\eta_1 \geq \eta_2 \geq 0$.
\item\label{t2}\label{t3} For any derivative $K' \subset \R^m$ of $K$, any proper measure $\mi$ on $K'$ and any coordinate-wise decomposition $\R^m = \R^k \times \R^{m-k}$ the function 
$$\theta_k^\mi(\mathbf{x}) = \frac{\int_{\RO^k} \eta_2((\mathbf{y},\mathbf{x})) d\mi_{|\R^k}(\mathbf{y})}{\int_{\RO^k} \eta_1((\mathbf{y},\mathbf{x})) d\mi_{|\R^k}(\mathbf{y})}$$
is a coordinate-wise non-increasing function of $\mathbf{x} = (x_{j_1},\ldots,x_{j_{m-k}})$ where defined. Recall $\mi_{|\R^k}$ denotes the restriction of $\mi$ to $\R^k$.
\end{pasztetowa}

For a fixed proper measure $\mi$ on $K$ we define the function $\theta^{\mi}$ by $$\theta^{\mi}(A) = \frac{\int_A \eta_2(x) d\mi(x)}{\int_A \eta_1(x) d\mi(x)}$$ for any Borel set $A$ with $\mi(A) > 0$. We shall say that $\theta^{\mi}$ is the $\Theta$ function defined for the measure $\mi$ by $\eta_1$ and $\eta_2$.

For a fixed proper measure $\mi$ on $K$ and a fixed coordinate-wise decomposition $\R^n = \R^k \times \R^{n-k}$ we shall also define $$\theta_{n-k}^\mi(a_1,a_2,\ldots,a_k;A) = \frac{\int_A \eta_2(a_1,a_2,\ldots,a_k,x_{k+1},x_{k+2},\ldots,x_n) d\mi_{|\{(a_1,a_2,\ldots,a_k)\} \times \R^{n-k}}}{\int_A \eta_1(a_1,a_2,\ldots,a_k,x_{k+1},x_{k+2},\ldots,x_n) d\mi_{|\{(a_1,a_2,\ldots,a_k)\} \times \R^{n-k}}}$$ for such sets $A$ and number $a_1,a_2,\ldots,a_k$ for which the denominator is positive. If $A = \R^{n-k}$ we shall omit it and write $\theta_{n-k}^\mi(a_1,a_2,\ldots,a_k)$ for $\theta_{n-k}^\mi(a_1,a_2,\ldots,a_k;\R^{n-k})$, and if $\mathbf{a} = (a_1,a_2,\ldots,a_k)$, we will write $\theta_{n-k}^\mi(\mathbf{a};A)$ or $\theta_{n-k}^\mi(\mathbf{a})$ for $\theta_{n-k}^\mi(a_1,a_2,\ldots,a_k;A)$ and $\theta_{n-k}^\mi(a_1,a_2,\ldots,a_k;\R^{n-k})$ respectively, which is consistent with the notation above. If $A \subset \R^n$ by $\theta^\mi_{n-k}(\mathbf{a};A)$ we mean $\theta^\mi_{n-k}(\mathbf{a};A \cap \{\mathbf{a}\} \times \R^{n-k})$. If there could be doubts as to what coordinate-wise decomposition is taken, we may write $\theta_{n-k}^\mi(x_1 = a_1,x_1 = a_2,\ldots,x_k = a_k;A)$ for $\theta_{n-k}^\mi(a_1,a_2,\ldots,a_k;A)$.
\end{Def}

\begin{fact} \label{wlasnoscitheta} If $\eta_1$ and $\eta_2$ define a $\Theta$ function $\theta^\mi$ for a proper measure $\mi$ on a generalized Orlicz ball $K$, then the following are true:
\begin{pasztetowa}\polarbuza{T}
\item\label{t6} The function $\theta^\mi$ is continuous with respect to the symmetric difference distance, that is if $\theta^\mi$ is defined for all $C_i$ and $\mi(C_0 \bigtriangleup C_i) \ra 0$, then $\theta^\mi(C_i) \ra \theta^\mi(C_0)$.
\item\label{t7} If $D' \subset D \subset \R^n$, $\theta^\mi(D) = \theta^\mi(D')$ and $\theta^\mi$ is defined for $D \setminus D'$, then $\theta^\mi (D\setminus D') = \theta^\mi(D)$.
\item\label{t8} If $K'$ is a derivative of $K$, then the restrictions of $\eta_1$ and $\eta_2$ to $K'$ define a $\Theta$ function on $K'$.
\end{pasztetowa}
\end{fact}

Further on, as the proper measure taken rarely changes, we omit the $\mi$ in the upper index and simply write $\theta$ for $\theta^\mi$. Note that as $\eta_1$ is positive on $\KO$ from property (\ref{t1}) and $\supp \mi \subset \KO$, we know that $\theta^\mi(D)$ is well defined if and only if $\mi(D) > 0$. 

\begin{Def} Functions $\eta_1$ and $\eta_2$ defining a $\Theta$ function on a generalized Orlicz ball $K$ are said to define a strict $\Theta$ function if the following extra conditions are satisfied:
\begin{enumerate}\polarbuza{S}
\item\label{s2} $\supp \eta_2 \subset \Int_{\RO} \supp \eta_1$, where $\Int_{\RO}$ denotes the interior taken with respect to the space $\RO$.
\item\label{s3} The generalized Orlicz ball $K$ is proper.
\item\label{s4} For any coordinate-wise decomposition $\R^n = \R^k \times \R^{n-k}$ with $n > k$ the functions $\teta_i : \R^k \ra \R$ defined by $x \mapsto \int_{\R^{n-k}} \eta_i(x,y) d\lambda_{n-k}(y)$ are continuous. % $\eta_i$ are continuous functions.
\item\label{t4} $\eta_1 > 0$ on $\Int K_+$.

%% ??? Dwie zmiany --- dodaje tu warunek $\eta_1 > 0$, kt/ory dotychczas by/l kiepski, i zmieniam ci/ag/lo/s/c $\eta$ na ci/ag/lo/s/c ca/lek $\teta_i$.

%% (TO STARY WARUNEK --- RAFAL TWIERDZI, ZE GDZIES W SRODKU INDUKCJI POZASKONCZONEJ KORZYSTALEM Z CIAGLOSCI \eta, A TO MOZNA STOSUNKOWO BEZ BOLU ZALOZYC, A PRZYNAJMNIEJ TAK MI SIE WYDAJE) For any coordinate-wise decomposition $\R^n = \R^k \times \R^{n-k}$ the function $\theta_{n-k} : \R^{n-k} \ra [0,1]$ defined by $$\theta_k(\mathbf{x}) = \frac{\int_{\R^k} \eta_2((\mathbf{y},\mathbf{x})) d\mi_{|\R^k}(\mathbf{y})}{\int_{\R^k} \eta_1((\mathbf{y},\mathbf{x})) d\mi_{|\R^k}(\mathbf{y})}$$ is an continuous function of $\mathbf{x} = (x_{j_1},\ldots,x_{j_{n-k}})$.
\end{enumerate}
\end{Def}

\begin{Def} For $\eta_1$ and $\eta_2$ defining a $\Theta$ function $\theta$ on a generalized Orlicz ball $K$ by a {\em derivative} of $\theta$ we mean the function defined on a derivative $K'$ of $K$ by the restrictions of $\eta_1$ and $\eta_2$ to $K'$. Note that the derivatives of a $\Theta$ function are $\Theta$ functions.\end{Def}

\begin{Def} For a given generalized Orlicz ball $K$ we say that $\eta_1$ and $\eta_2$ define a {\em weakly non-degenerate $\Theta$ function} on $K$ if for every $\eps > 0$ there exists a generalized Orlicz ball $K' \subset K$ with $\lambda(K \setminus K') \leq \eps \lambda(K)$ and functions $\eta_1'$ and $\eta_2'$ defining a strict $\Theta$ function on $K'$ with $\int|\eta_i - \eta_i'| d\lambda \leq \eps$. A $\Theta$ function is called {\em non-degenerate} if it is weakly non-degenerate and all its derivatives are weakly non-degenerate.\end{Def}

Note that as the density of any proper measure is bounded, in all the bounds in the definition above we can replace $\lambda$ by any proper measure $\mi$.

Frequently we shall take the same collection of assumptions for our theorems. To make reading the paper easier, we will use the following notation:
\begin{Def} We shall speak of
\begin{itemize}
\item {\em Standard assumptions} if $K \subset \R_x \times \R_y \times \R^{n-2}$ is a generalized Orlicz ball, $\mi$ is a proper measure for $K$, $\eta_1$ and $\eta_2$ define a $\Theta$ function $\theta = \theta^\mi$ on $K$ for $\mi$ and $A$ is a c-set in $\RO^n$,
\item {\em Non-degenerate assumptions} if additionally we require the $\Theta$ function defined by $\eta_1$ and $\eta_2$ to be non-degenerate, and
\item {\em Strict assumptions} if $K$ is a proper generalized Orlicz ball and $\eta_1$ and $\eta_2$ define a strict non-degenerate $\Theta$ function.
\end{itemize}
\end{Def}

\begin{Def} \label{appropriate} Under standard assumptions
% For a given generalized Orlicz ball $K$ with a proper measure $\mi$, c-set $A$ and functions $\eta_1$, $\eta_2$ defining a $\Theta$ function $\theta$ 
a set $D \subset \R^n$ will be called {\em appropriate}, if 
\begin{itemize}
\item $\theta(D)$ is defined,
\item $\theta(D \cap A) \geq \theta(K) \geq \theta(D \cap \bA)$ if the left-hand side and the right-hand side are defined,
\item $\theta(D) = \theta(K)$.
\end{itemize}
\end{Def}

\begin{Def} \label{epsappropriate} Under standard assumptions 
% Let $K \subset \R_x \times \R_y \times \R^{n-2}$ be a generalized Orlicz ball with a proper measure $\mi$. Let $A$ be a c-set, let $\eta_1$ and $\eta_2$ define a $\Theta$ function $\theta$ and 
let $\mi_2$ be the restriction of $\mi$ to $\R_x \times \R_y \times \{0\}$. For any $\eps > 0$ a set $\tD \times \R^{n-2} = D$ is called {\em $\eps$-appropriate}, if 
\begin{itemize}
\item $\theta(D)$ is defined,
\item $\theta(D) = \theta(K)$,
\item For each $U \in \{A, \bA\}$ and each $i \in \{1,2\}$ there exists a number $C_{U,i}$ such that $$\bigg|\int_{D\cap U} \eta_i(t) d\mi(t) - C_{U,i}\bigg| \leq \eps \mi_2(\tD)$$ and $$\frac{C_{A,2}}{C_{A,1}} \geq \theta(D) \geq \frac{C_{\bA,2}}{C_{\bA,1}}.$$
\end{itemize}
\end{Def}
    
The definition of an appropriate set describes the properties we desire for the set into which we divide $\KO$. In fact, due to the approximation, we shall divide $\KO$ into $\eps$-appropriate sets to prove it is $\eps$-appropriate, and then take $\eps \ra 0$.

\subsection{The generalized Orlicz ball lemmas}

In this subsection we will prove a few lemmas about the structure generalized Orlicz balls. They show that the class of generalized Orlicz balls is closed under taking derivatives, and that proper generalized Orlicz balls are, in a sense, dense in the class of generalized Orlicz balls. These lemmas are the main reason the whole reasoning in this paper has to be done for generalized Orlicz balls, and not simply Orlicz balls --- the class of Orlicz balls does not enjoy the same closedness propeties.

\begin{fact} A product of intervals $\prod_{i=1}^n [a_i,b_i]$ is isometric to the positive quadrant of the Orlicz ball $K \subset \R^k$ defined by the functions \begin{eqnarray*} f_i(x_i) = \begin{cases} 0 & \hbox{ if $x_i \leq b_i - a_i$} \\ \infty & \hbox{ if $x_i > b_i - a_i$}\end{cases}\end{eqnarray*} for $b_i > a_i$.\end{fact}

\begin{lemma} \label{obciacboki} If $\KO \subset \R^n$ is a generalized Orlicz ball positive quadrant and $0 \leq x_a < x_b$, then $\KO \cap \{x_1 \in [x_a,x_b]\}$ is isometric to a generalized Orlicz ball positive quadrant or empty \end{lemma}

\begin{proof}
Let $f_1, f_2, \ldots, f_n$ be the Young functions defining $K$. Let $\KO' = \KO \cap \{x_1 \in [x_a,x_b]\}$. Let $c = f_1(x_a)$. If $c = 1$, then $\KO' = \{x : x_1 = x_a, \forall_{i > 1} f_i(x_i) = 0\}$, which is a product of intervals and thus isometric to a generalized Orlicz ball positive quadrant. If $c > 1$ then $\KO'$ is empty. If $c < 1$ we define $\bar{f}_1$ by 
\begin{eqnarray*}\bar{f}_1(x_1) = \begin{cases} \frac{f_1(x_1 + x_a) - f_1(x_a)}{1-c} & \hbox{ for $x_1 < x_b$} \\
                      \infty                                     & \hbox{ for $x_1 > x_b$,}\end{cases}\end{eqnarray*}
and $\bar{f}_i$ for $i > 1$ by 
$$\bar{f}_i(x_i) = \frac{f_i(x_i)}{1-c}.$$ Now $(x_1, x_2, \ldots, x_n) \in \bar{K}_+$ iff $(x_1 + x_a, x_2, \ldots, x_n) \in \KO'$, where $\bar{K}_+$ is the positive quadrant of the Orlicz ball defined by $\bar{f}_i$.
\end{proof}

\begin{lemma}\label{sekcjajest}
If $K \subset \R^n$ is a generalized Orlicz ball and $H = \{x\in \R^n : x_1 = \lambda x_2 + c\}$ is a positively inclined hyperplane in $\R^n$, then $\KO \cap H$ is the positive quadrant of some generalized Orlicz ball $L$ or an empty set.\end{lemma}

\begin{proof}
As $H$ is positively inclined, $\lambda \geq 0$. If $\lambda = 0$ and $c < 0$ we have $\KO \cap H = \emptyset$. If $c < 0$ and $\lambda > 0$ we can transform the equation giving $H$ to $H = \{x \in \R^n : x_2 = \frac{1}{\lambda}x_1 - \frac{c}{\lambda}\}$. Thus we can assume $c \geq 0$.

For $x_2 \geq 0$ we have $x_1 \geq c$ in $H$. Thus if $f_1(c) > 1$, then for $x_2 \geq 0$ we have $f_1(x_1) > 1$ for $x \in H, x_1 \geq 0$, thus $H \cap \KO = \emptyset$. If $f_1(c) = 1$ and $\lambda > 0$, then $H \cap \KO$ is the set $\{x : x_1 = c, x_2 = 0, f_i(x_i) = 0$ for $i > 2\}$. This set is a cartesian product of intervals, and isometric to a generalized Orlicz ball positive quadrant. If $f_1(c) = 1$ and $\lambda = 0$, the situation is the same, except $x_2 = 0$ is replaced by $f_2(x_2) = 0$.

Now we may assume $f_1(c) < 1$. Let $x_3,x_4,\ldots,x_{n+1}$ be the coordinates on $H$, with $x_1 = \lambda x_{n+1} + c$, $x_2 = x_{n+1}$. Let us take $f_{n+1}(t) = f_1(\lambda t + c) + f_2(t) - f_1(c)$, then $f_1(x_1) + f_2(x_2) = f_{n+1}(x_2) + f_1(c)$. The function $f_{n+1}$ is a sum of three convex functions, thus it is convex, and $f_{n+1}(0) = 0$. The set $\{(x_i)_{i=3}^{n+1} : f_i(x_i) < 1 -f_1(c)\} \cap \RO$ is equal to $\KO \cap H$. If we consider Young functions $\tilde{f}_i(t) = \frac{f_i(t)}{1 - f_1(c)}$ for $i = 3,4,\ldots,n+1$ we get the generalized Orlicz ball $L \subset H$ such that $\KO \cap H = L \cap \RO^{n}$.\end{proof}

\begin{lemma} \label{approrlicz}
For any generalized Orlicz ball $K\subset \R^n$ and any $\eps > 0$ there exists a proper generalized Orlicz ball $K' \subset K$ with $\lambda(K \setminus K') < \eps$. Furthermore if any Young function $f_i$ of $K$ is already a proper Young function, the same $f_i$ will be the appropriate Young function of $K'$. 
\end{lemma}

\begin{proof}
This lemma is easy to believe in, but somewhat technical to prove. An impatient reader might be well advised to skip the next two proofs (or prove the Lemmas her$\slash$himself, if desired) and go to the more crucial parts of the paper.

As any generalized Orlicz ball is 1-symmetric, it suffices to prove $\lambda(\KO \setminus K'_{\geq 0}) \leq \eps \slash 2^n$. We shall thus consider only the points in $\RO^n$ and decrease $\eps$ to be $2^n$ times smaller. Recall that a proper Young function is such a Young function that $f(x) = 0$ only for $x = 0$ and $f(x) < \infty$. Thus we have to get rid of superfluous zeroes and of infinity values. First we shall take care of the zeroes.

Let $f_i$ be Young functions defining $K$. Let $M$ be the largest of the $(n-1)$-dimensional measures of the projections of $K$ onto the hyperplanes $x_i = 0$. Let $t_i = \inf\{x_i : f_i(x_i) > 1\slash 2n\}$. Let $c = \inf_i \{f_i'(t_i)\}$. We shall prove that for $\delta < 1\slash 2n$ the set $U_\delta = \{x : \sum f_i(x_i) \in [1-\delta,1]\}$ has measure no larger than $\frac{M n \delta}{c}$.

First note that $U_\delta = \bigcup U_i$, where $U_i = U_\delta \cap \{x : f_i(x_i) > 1\slash 2n\}$, as at least one of $f_i(x_i)$ has to be large for the sum to be large. We shall bound the measure of each $U_i$ separately. For each point $x = (x_1,\ldots,x_{i-1},x_{i+1},\ldots,x_n)$ the set of those $x_i$ that $(x,x_i) \in U_i$ has length at most $\frac{\delta}{c}$. Thus, from Fubini's theorem, the measure of $U_i$ can be bounded by $\frac{M \delta}{c}$, and summing over all $i$ we get the desired bound for $U_\delta$.

Let us take $\delta = \frac{c\eps}{2Mn^2}$. For each $i$ for which we have superfluous zeroes let us take $s_i = \inf\{x_i : f_i(x_i) > \delta\}$ and replace $f_i$ by $g_i$ defined by 
$$g_i(x_i) = \begin{cases} 
f_i(x_i) &\hbox{ for }x_i \geq s_i \\ 
x_i \delta \slash s_i &\hbox{ for }x_i < s_i.\end{cases}$$ 
We have $g_i(x_i) \geq f_i(x_i)$ and $g_i(x_i) - f_i(x_i) \leq \delta$. Thus if $K'$ is the generalized Orlicz ball defined by $g_i$, we have $K' \subset K$ and $K \setminus K' \subset U_{n\delta}$, and thus $\lambda(K \bigtriangleup K') \leq \frac{n^2 M \delta}{c} = \eps \slash 2$. 

Now we shall deal with the $\infty$ values. Let $\delta = \frac{\eps}{2nM}$. Note that the shape of $K'$ is determined by the values of $g_i$ only up to $g_i(x_i) = 1$. Thus we have to make some corrections to $g_i$ up to $g_i(x_i) = 1$, and then extend $g_i$ anyhow, say linearly.
For each $i$ such that $g_i$ attains the $\infty$ value let $r_i = \inf \{x_i : g_i(x_i) = \infty\}$, and let $v_i = \lim_{x_i \ra v_i^-} g_i(x_i)$. If $v_i \geq 1$, then all we have to do is to extend $g_i$ in a different way after $r_i$, and that does not change the ball $K'$ defined by $g_i$. If, however, $v_i < 1$, we define $h_i$ as follows:
$$h_i(x_i) = \begin{cases}
g_i(x_i) &\hbox{ for }x_i < r_i - \delta \\
2 &\hbox{ for }x_i = r_i \\
\hbox{linear continuous extension} &\hbox{ otherwise.}
\end{cases}$$
Let $K''$ be the ball defined by $h_i$. Again, $K'' \subset K'$, as $h_i \geq g_i$ on the set where $g_i \geq 1$, from the convexity of $g_i$. The difference, however, is obviously contained in $\bigcup_i K' \cap \{x_i \in [r_i - \delta,r_i]\}$, thus $\lambda(K' \setminus K'') \leq n M \delta = \eps \slash 2$. Adding the two estimates together we get $\lambda(K \setminus K'') \leq \eps$. 
\end{proof}

\begin{cor} \label{apprsection} With the assumptions of Lemma \ref{approrlicz} if we take any $y_0$ (where $y$ is any coordinate in $\R^n$), then we can take such a $K'$ as before and $y_1$ that $\lambda_{n-1} ((K \cap \{y = y_0\}) \bigtriangleup (K' \cap \{y = y_1\})) < \eps$.\end{cor}

\begin{proof}
This, again, is easy to believe in, and actually simple if $f_y(y_0) \neq 1$. The special case where $f_y(y_0) = 1$ could arguably be ignored (as it happens only on a set of measure zero), but to avoid omitting a set of measure zero in all other places of the proof, we shall go through the technicalities here.

If $f_y(y_0) > 1$, we can simply take $y_1 = y_0$, and $\lambda_{n-1} (K \cap \{y = y_0\}) = \lambda_{n-1} (K' \cap \{y = y_1\}) = 0$. 

If $f_y(y_0) < 1$, we need to control the Orlicz ball $\sum f_i(x_i) = 1 - f_y(y_0)$. This Orlicz ball $L$ is given by Young functions $f_i \slash (1- f_y(y_0))$. For this Orlicz ball we also calculate values of $M$ and $c$, and apply the reasoning in the proof of Lemma \ref{approrlicz} taking the larger $M$ and the smaller $c$ of those calculated for the two balls. We thus get good approximations $K'$ and $L'$ of both $K$ and $L$. Now take such a $y_1$ that $f_y(y_0) = h_y(y_1)$, this can be done as $h_y$ is continuous. Now $K' \cap \{y = y_1\} = L'$, which proves the thesis.

In the case $f_y(y_0) = 1$ if any of the other $f_i$ do not have superfluous zeroes, the measure of $K \cap \{y = y_0\}$ is 0, and thus taking $y_1 = y_0 + 1$ we get the thesis. If, however, all the other $f_i$ have superfluous zeroes, the intersection $K \cap \{y = y_0\}$ is the cube $\bigtimes_i f_i^{-1}(0)$. In this case we shall need a better approximation. Let $z_i = \sup \{x_i : f_i(x_i) = 0\}$. Let us, as before for $\eps$, define $\delta' = \frac{c\eps'}{2Mn^2}$ and $s_i' = \inf\{x_i : f_i(x_i) > \delta'\}$. We need $\eps'$ to be so small that $$s_i' \slash z_i \leq 
1 + \frac{(1 + \eps \slash M)^{1\slash n} - 1}{n}$$ and smaller than $\eps$. Note that as $\eps' \ra 0$ we have $\delta' \ra 0$ and $s_i' \ra z_i$, so taking $\eps'$ small enough we can achieve the desired inequality for all $i$. Conduct the proof of Lemma \ref{approrlicz} taking $\eps'$  instead of $\eps$. Take $y_1$ such that $h_y(y_1) = 1 - n\delta'$. Note that if $x_i \leq s_i'$ for all $i$, then $\sum h_i(x_i) \leq n \delta'$, and thus $x = (x_i) \in K' \cap \{y = y_1\}$. On the other hand if for any $i$ we have $x_i > s_i' + (n-1) (s_i' - z_i)$, then $h_i(x_i) \geq f_i(x_i)$, and as $f_i(z_i) = 0$, $f_i(s_i') \geq \delta'$ and $f_i$ is convex, we have $f_i(x_i) > n\delta'$, and thus $\sum h_i(x_i) > n\delta'$ and $x \not\in K' \cap \{y = y_1\}$. Thus $$K \cap \{y = y_0\} = \bigtimes_i f_i^{-1}(0) \subset K' \cap \{y = y_1\} \subset \bigtimes_i [0,z_i + n (s_i' - z_i)].$$ Now we have the following inequalities:
\begin{align*}
\frac{s_i'}{z_i} &\leq 1 + \frac{(1 + \eps \slash M)^{1\slash n} - 1}{n} \\
\frac{n(s_i' - z_i)}{z_i} & \leq (1 + \eps \slash M)^{1\slash n} - 1 \\
\frac{z_i + n(s_i' - z_i)}{z_i} &\leq (1 + \eps \slash M)^{1\slash n} \\
\prod_i \frac{z_i + n(s_i' - z_i)}{z_i} &\leq 1 + \eps \slash M \\
\lambda (\bigtimes_i [0,z_i + n(s_i' - z_i)]) &\leq \lambda (\bigtimes [0,z_i]) + \frac{\eps \lambda (\bigtimes[0,z_i])}{M} \\
\lambda (K' \cap \{y = y_1\}) - \lambda(K \cap \{y = y_0\}) &\leq \eps.
\end{align*}
The last inequality follows as $\bigtimes[0,z_i]$ is a subset of the projection of $K$ onto $y = 0$, and thus its measure is no bigger than $M$. This, along with the fact that $K \cap \{y = y_0\} \subset K' \cap \{y = y_1\}$ gives the thesis.

This reasoning can be extended to approximate any finite number of sections of $K$ along with $K$.
\end{proof}

\subsection{$1\slash m$-concave functions and proper measures}

Here we give a few elementary facts about $1\slash m$-concave functions and proper measures. Most facts are easily proved and quite a few are well known, so we skip some of the proofs.

\begin{fact} If a function $f$ is $1\slash m$-concave for some $m > 0$, then it is also $1 \slash m'$-concave for any $m' > m$.\end{fact}

\begin{fact} \label{prodncon} The product of $1\slash m$-concave functions is $1\slash 2m$-concave.\end{fact}

\begin{fact} \label{convconc}From the Brunn-Minkowski inequality, if $K \subset \R^n$ is a convex set, then $(y_1,y_2,\ldots,y_k) \mapsto \lambda_{n-k}(K \cap \{\forall_{1\leq i \leq k} x_i = y_i\})$ is a $1\slash (n-k)$-concave function, where $x_i$ are the coordinates on $\R^n$. Conversely, if we have a $1\slash m$-concave function on $\R^n$, then there exists a convex set $K \subset \R^{n+m}$ such that $f$ is the projection of the Lebesgue measure restricted to $K$ onto $\R^n$. As a corollary of these two facts the projection of a $1\slash m$ concave function on $\R^n$ onto $\R^k$ is a $1\slash(n+m-k)$-concave function.\end{fact}

\begin{fact} The restriction of the Lebesgue measure to $\KO$ is a proper measure with respect to $K$.
\end{fact}

\begin{fact} The support of a proper measure $\mi$ is a convex set.\end{fact}

\begin{lemma} \label{sekcjami} If $\mi$ is a proper measure on $\R^n$ and $H = \{x\in \R^n : x_1 = \lambda x_2 + c\}$, with
$\lambda \geq 0$ is a hyperplane in $\R^n$, then $\mi$ restricted to $H$ with coordinates $(v,x_3,\ldots,x_n)$ is a proper measure.\end{lemma}

\begin{proof} The density of $\mi$ is equal to $f(x) g(y) \1_K$ for some $(1\slash m)$-concave functions $f$ and $g$ and some generalized Orlicz ball $K$. From Lemma \ref{sekcjajest} the set $\KO \cap H$ is the positive quadrant of some Orlicz ball $K'$ with coordinates $(v,x_3,\ldots,x_n)$. The product $f \cdot g$ on $H$ varies with at most two variables, and is from, Fact \ref{prodncon}, a $(1\slash 2n)$-concave function with respect to $v$.
\end{proof}

\begin{lemma}\label{cieciemi} If $\mi$ is a proper measure on $\R_x \times \R_y \times R^{n-2}$, then the restriction of $\mi$ to an interval $I$ with respect to any variable is also a proper measure.\end{lemma}

\begin{proof} Due to Lemma \ref{obciacboki} if $\KO$ is the Orlicz ball quadrant for which $\mi$ is defined, $\KO' = \KO \cap \{t \in I\}$ is also an Orlicz ball quadrant. Let $f$ and $g$ be the functions defining the density of $\mi$. To make them define a proper measure on $\KO'$ we simply have to restrict them to the set $\{x_0 : \lambda_{n-1} (K'_{x = x_0}) > 0\}$ for $f$ and similarly for $g$, and additionaly to the interval $I$ if it was taken in $x$ or $y$ respectively. Both functions will have a convex support after this restriction, and as they were $1\slash m$-concave on a larger domain, they will still be $1\slash m$-concave.
\end{proof}

\subsection{Lens sets and $\Theta$ functions}

\begin{fact} Let $D$ be a lens set with extremal points $(x_1,y_1)$ and $(x_2,y_2)$. Let $x^-(y) = \inf \{x : (x,y) \in D\}$, $x^+(y) = \sup \{x : (x,y) \in D\}$ for $y \in [y_1,y_2]$. Then $x^-$ and $x^+$ are increasing function on their domains, $x^-$ is convex, and $x^+$ is concave.\end{fact}

\begin{lemma} \label{moveinterval} Under standard assumptions consider a fixed $y_0$ and two intervals $[x_a,x_b], [x_c,x_d]$ with $x_a \leq x_c$ and $x_b \leq x_d$. Then we have $$\theta^\mi_{n-1}(y_0; [x_a,x_b] \times \R^{n-2}) \geq \theta^\mi_{n-1}(y_0; [x_c,x_d] \times \R^{n-2})$$ if both sides are defined.

%Moreover, if equality occurs and $\theta^\mi_{n-1}(y_0; ([x_a,x_c] \cup [x_b,x_d]) \times \R^{n-2})$ is defined, then $\theta^\mi_{n-2}(x, y_0)$ is constant as a function of $x$ where defined.
%Wypierdzielam czesc moreover, bo nigdzie nie jest wykorzystywana, a za to jest falszywa jak [ac] \cap Dom theta = [bd] \cap Dom theta

The same applies when $x$ is exchanged with $y$.\end{lemma}

\begin{proof} 
From property (\ref{t2}) we know that $\theta^\mi_{n-2}(x,y_0)$ is a decreasing function of $x$. The domain of this function is a convex set, so its intersections with $[x_a,x_b]$ and $[x_c,x_d]$ are both intervals (they are non-empty, for $\theta^\mi_{n-1}$ is defined for both intervals). Applying Lemma \ref{rosnmix}, part 2, to $\int_{\R^{n-2}} \eta_2(x,y,t_1,\ldots,t_{n-2}) d\mi_{|(x,y)\times \R^{n-2}} (t_1,\ldots,t_{n-2})$ and $\int_{\R^{n-2}} \eta_1(x,y,t_1,\ldots,t_{n-2}) d\mi_{|(x,y)\times \R^{n-2}} (t_1,\ldots,t_{n-2})$ and the shortened intervals we get the thesis.
\end{proof}

\begin{lemma} \label{raiseinterval} Under standard assumptions for a given interval $I = [x_a,x_b]$ the function $\theta^\mi_{n-1}(y;I \times \R^{n-2})$ is a decreasing function of $y$ on its domain. The same applies when $x$ is exchanged with $y$.\end{lemma}
\begin{proof}
Take any $0 \leq y_1 \leq y_2$ in the domain. $K' = K \cap \{x \in I\}$ is a derivative of $K$, and $\theta^\mi_{n-1}(y;I\times \R^{n-2}) = \bar{\theta}_{n-1}^{\mi}(y)$, where $\bar{\theta}^\mi$ is defined by the restrictions of $\eta_1$ and $\eta_2$ to $K'$. Thus from property (\ref{t3}) we get the thesis.
\end{proof} 

\begin{lemma} \label{incthetay} Under standard assumptions for a given lens set $D$ the domain of the function $y \mapsto \theta^\mi_{n-1}(y;D)$ is an interval and the function is decreasing.
\end{lemma}

\begin{proof}
Let $(x_1,y_1)$ and $(x_2,y_2)$ be the extremal points of $\tD$. Take any $y_4$ such that $\theta_{n-1}^\mi(y_4;D)$ is defined and take any $y_3 \in (y_1,y_4]$. Thus $\theta_{n-2}^\mi(x,y_4)$ is defined for more than one $x$ such that $(x,y_4) \in \tD$ (actually, for a set of positive Lebesgue measure), let $x_4$ be any such $x$ except the smallest. We want to prove $\theta_{n-1}^\mi(y_3;D)$ is defined. Note that $\supp\mi$ is a c-set on $x > x_-$, $y > y_-$ and $\supp \eta_1$ is also a c-set, thus their intersection is a c-set. Thus $\theta_{n-2}^\mi(x,y)$ is defined for any $x_- < x \leq x_4$ and $y_- < y \leq y_4$. As $D$ is a lens set, the set of $x \leq x_4$ such that $(x,y_3) \in \tD$ has positive Lebesgue measure, thus $\theta_{n-1}^\mi(y_3;D)$ is defined, which means that $\theta_{n-1}^\mi(y;D)$ is defined on some interval $(y_1,y_0)$ and undefined outside.

Now we shall prove $\theta_{n-1}^\mi(y;D)$ is decreasing. Take $y_3 \leq y_4$ from the domain. Let $[x_3^-,x_3^+]$ be the interval $\tD \cap \{y = y_3\}$ and $[x_4^-,x_4^+]$ the interval $\tD \cap \{y=y_4\}$. From the definition of a lens set $x_3^- \leq x_4^-,x_3^+ \leq x_4^+$. 
From Lemmas \ref{raiseinterval} and \ref{moveinterval} (twice) we have
\begin{align*} \theta_{n-1}^\mi(y_3;D) & = \theta_{n-1}^\mi(y_3;[x_3^-,x_3^+] \times \R^{n-2}) \geq \theta_{n-1}^\mi(y_3;[x_3^-,x_4^+] \times \R^{n-2}) \\  & \geq \theta_{n-1}^\mi(y_4;[x_3^-,x_4^+]\times \R^{n-2}) \geq \theta_{n-1}^\mi(y_4;[x_4^-,x_4^+] \times \R^{n-2}) = \theta_{n-1}^\mi(y_4;D).\end{align*}
Note that the last expression in the first line and the first in the second line are well defined, for the second argument is a superset of the second argument for $\theta_{n-1}^\mi(y_3;D)$ and $\theta_{n-1}^\mi(y_4;D)$ respectively.
\end{proof}

\begin{cor}\label{thetamain}
Under standard assumptions for a given lens set $D$ and a given $y_0$ in the domain of $\theta_{n-1}^\mi(y;D)$ we have $$\theta^\mi(D \cap \{y \leq y_0\}) \leq \theta_{n-1}^\mi(y_0;D) \leq \theta^\mi(D \cap \{y \geq y_0\})$$ and $$\theta^\mi(D \cap \{y \leq y_0\}) \leq \theta^\mi(D) \leq \theta^\mi(D \cap \{y \geq y_0\})$$ if the left and right hand sides are defined. 

Moreover, if $\theta^\mi(D \cap \{y \leq y_0\}) = \theta^\mi(D \cap \{y \geq y_0\})$ for any $y_0 \in (y_1,y_2)$, then $\theta^\mi_{n-1}(y_3;D), \theta^\mi(D \cap \{y \geq y_3\})$ and $\theta^\mi(D \cap \{y \leq y_3\})$ are all constant where defined and equal $\theta^\mi(D)$ for $y_3 \in (y_1,y_2)$.
\end{cor}

\begin{proof}
From Lemma \ref{incthetay} the function $\theta_{n-1}^\mi(y;D)$ is decreasing as a function $y$ on its domain, and its domain is an interval. We know that $\supp \eta_2 \subset \supp \eta_1$, so we can apply Lemma \ref{rosncalk}, part 1a, to the appropriate integrals of $\eta_2$ and $\eta_1$ to get the first part of the thesis. The second part follows from the first and Fact \ref{obvi}. The third follows again from Lemma \ref{rosncalk}, part 1b.
\end{proof}

\begin{prop} \label{epstotot}
Under standard assumptions if $D$ is $\eps$-appropriate for any $\eps > 0$, then $D$ is appropriate.
\end{prop}

\begin{proof}
The third and first condition in Definition \ref{appropriate} follows from the definition of $\eps$-appropriate for any $\eps$. We have to check the second condition. Let $C_{U,i}^\eps$ denote the numbers $C_{U,i}$ which show $D$ is and $\eps$-appropriate set. We have 
$$\theta^\mi(D) = \frac{C_{A,2}^\eps}{C_{A,1}^\eps} \leq \frac{\int_{D\cap A} \eta_2(t) d\mi(t) + \eps \mi_2(\tD)}{\int_{D \cap A} \eta_1(t) d\mi(t) - \eps \mi_2(\tD)} \rightarrow_{\eps \ra 0} \theta^\mi(D\cap A),$$ and similarly for the second inequality.
\end{proof}

\begin{prop} \label{epssum}
Under standard assumptions if $D_k$ are $\eps$-appropriate sets for $k \in K$, then $D = \bigcup_{k\in K} D_k$ is $\eps$-appropriate.
\end{prop}

\begin{proof}
We have $\theta^\mi(D) = \theta^\mi(K)$ from Lemma \ref{dividesets}. For the third condition in Definition \ref{epsappropriate} we take $C_{U,i}^D = \sum_{k\in K} C_{U,i}^{D_k}$. These are good approximations as $\mi_2(\tD) = \sum_{k} \mi_2(\tD_k)$, and obviously satisfy the proportion inequality.
\end{proof}

\section{The $\Theta$ theorem}
\subsection{Preparations for divisibility}
In this section we shall prove the main theorem concerning $\Theta$ functions. Under standard assumptions, we shall consider $\mi$ to be a fixed proper measure on $\R_x \times \R_y \times \R^{n-2}$. By $\mi_2$ we shall denote the restriction of $\mi$ to $\R_x \times \R_y \times \{0\}$. Note that as the support of $\mi$ is a c-set with respect to the $n-2$ variables of $\R^{n-2}$, the support of $\mi_2$ is the projection of the support of $\mi$. As we fix $\mi$, we shall omit the upper index when writing the $\Theta$ function and write $\theta$ or $\theta_k$ instead of $\theta^\mi$ or $\theta_k^\mi$.

The main theorem we want to prove is:

\begin{thm}\label{main}
Under non-degenerate assumptions $\theta(A) \geq \theta(K)$ and $\theta(K) \geq \theta(\bA)$, whenever both sides of an inequality are defined.
\end{thm}

This looks like a quite simple theorem, and we suspect there is a simpler proof than the one we present here. However, we were not able to find it (and would be interested to learn if anyone does). Notice that if $\theta(A)$ is undefined, then $\int_A \eta_1 d\mi = 0$, which implies $\int_A \eta_2 d\mi = 0$ from property (\ref{t1}). Thus $\theta(\bA) = \theta(K)$ and the Theorem is satisfied. Thus we assume $\theta(A)$ is defined. Similarly we may assume $\theta(\bA)$ is defined. From Fact \ref{obvi} it is enough to prove $\theta(A) \geq \theta(K)$ and the second inequality will follow. Thus, we concentrate on the first inequality. First, for technical reasons, we shall deal with the low-dimensional case:

\begin{thm}\label{lowdimmain} Under standard assumptions with $n \leq 2$ (that is, $K \subset \R$ or $K \subset \R^2$) we have $\theta(A) \geq \theta(K)$ and $\theta(K) \geq \theta(\bA)$, whenever both sides of an inequality are defined.\end{thm}

\begin{proof}
For $n = 1$ the set $A$ is one-dimensional, and thus (being a c-set) is an interval of the form $[0,a)$. We apply property (\ref{t2}) to $K' = K$, the measure $\mi$ and the decomposition $\R \times \{0\}$ and get that $\frac{\eta_2}{\eta_1}$ is a decreasing function. Thus from Lemma \ref{rosncalk}, part 1a, $\theta(A) \geq \theta(\bA)$ and the thesis follows from Fact \ref{obvi}.

For $n = 2$ we shall approximate the set $A$ by a $l$-stair set. The $l$-stair set is defined as follows:
\begin{Def} A $1$-stair set defined by $x_1 = 0$ and $a_1 \geq 0$ (denoted $A(x_1;a_1)$) is the empty set. A $l$-stair set defined by $0 = x_1 \leq x_2 \leq \ldots \leq x_l$ and $a_1 \geq a_2 \geq \ldots \geq a_l \geq 0$, denoted $A(x_1,x_2,\ldots,x_l;a_1,a_2,\ldots,a_l)$ is defined by $$A(x_1,x_2,\ldots,x_l;a_1,a_2,\ldots,a_l) = \Big(A(x_1,x_2,\ldots,x_{l-1};a_1,a_2,\ldots,a_{l-1}) \cap \{x \leq x_l\}\Big) \cup A(0,a_l).$$ That means that a $l$-stair set consists of $l$ steps, the $k$-th step goes from $x_k$ to $x_{k+1}$ (the last one goes all the way to infinity) at height $a_k$. A proper $l$-stair set is a $l$-stair set with $a_l = 0$\end{Def}

Notice that $\theta(A) = \theta(K \cap A)$ as $\supp \mi \subset K$. Thus we may assume $A \subset K$, and thus $A$ is bounded. Take $A_n = \{(x,y) : ([xn]\slash n,y) \in A\}$, where $[xn]$ denotes the integer part of $xn$. This is a proper stair set defined by $0,1\slash n, 2\slash n, \ldots$ (a finite sequence as $A$ is bounded) and the sequence $a_k = \sup\{y : (k\slash n, y) \in A\}$. Notice also $A_{2^n} \supset A_{2^{n+1}} \supset A$ and $\mi (A_{2^n} \setminus A) \ra 0$. Thus $\theta(A_{2^n}) \ra \theta(A)$, so it is enough to prove $\theta(A_{2^n}) \geq \theta(K)$ and go to the limit. Thus, instead of considering all c-sets we may restrict ourselves to proper $l$-stair sets.

The proof for $A$ being a proper $l$-stair set will be an induction upon $l$. For $l = 1$ the set $A$ is empty and the thesis is obvious. For $l = 2$ let $I = [0,x_2]$. From Lemma \ref{raiseinterval} the function $\theta_1(y;I)$ is decreasing where defined. Thus from Lemma \ref{rosncalk} we have $\theta(I \times [0,a_1]) \geq \theta_1(a_1;I) \geq \theta(I\times [a_1,\infty))$. Note that $A = I \times [0,a_1]$. Thus if $\theta(A) \geq \theta(K)$ or $\theta(A)$ is undefined, the thesis is satisfied. Otherwise, as $\theta(K) > \theta(I \times [0,a_1])$ we have $\theta(K) > \theta(I \times [a_1,\infty))$ if defined, and thus from Lemma \ref{dividesets} $\theta(K) > \theta(I \times \R_y)$. Now apply property (\ref{t2}) to $K' = K$, the decomposition $\R_x \times \R_y$ and the measure $\mi$ to get that $\theta_1(x)$ is a decreasing function. Again from Lemma \ref{rosncalk} and Lemma \ref{dividesets} this implies $\theta(I \times \R_y) \geq \theta(K)$, a contradiction. Thus the thesis is satisfied for $l=2$.

For larger $l$ let $I = [x_{l-1},x_l]$. Again from Lemma \ref{raiseinterval} the function $\theta_1(y;I)$ is decreasing. Thus \begin{equation} \label{numerequa} \theta(I \times [0,a_{l-1}]) \geq \theta(I \times [a_{l-1},a_{l-2}])\end{equation} if both are defined. Note $$A(x_1,x_2,\ldots,x_l;a_1,a_2,\ldots,a_l) \setminus (I \times [0,a_{l-1}]) = A(x_1,x_2,\ldots,x_{l-1}; a_1,a_2,\ldots,a_{l-2},0)$$ and $$A(x_1,x_2,\ldots,x_l;a_1,a_2,\ldots,a_l) \cup (I \times [a_{l-1},a_{l-2}]) = A(x_1,x_2,\ldots,x_{l-2},x_l; a_1,a_2,\ldots,a_{l-2},a_l).$$

Suppose $\theta(A) < \theta(K)$. If $\theta(I \times [0,a_{l-1}]) > \theta(A)$ or is undefined, then from Lemma \ref{dividesets} we have $\theta(A \setminus (I \times [0,a_{l-1}])) \leq \theta(A)$, but from the inductive assumption for $l-1$ we have $\theta(A \setminus (I \times [0,a_{l-1}])) \geq \theta(K)$, from which $\theta(A) \geq \theta(K)$. If, on the other hand, $\theta(I \times [0,a_{l-1}]) \leq \theta(A)$, then from (\ref{numerequa}) $\theta(I \times [a_{l-1},a_{l-2}]) \leq \theta(A)$ or is undefined, thus from Lemma \ref{dividesets} $\theta(A) \geq \theta(A \cup (I \times [a_{l-1},a_{l-2}]))$, and again from the inductive assumption $\theta(A) \geq \theta(K)$.

Thus for any $l$ and for any $A$ being a proper $l$-stair set we have $\theta(A) \geq \theta(K)$, which ends the proof.
\end{proof}

\begin{proof}[Proof of the Theorem \ref{main}]
The proof will proceed by induction upon $n$. For $n \leq 2$ we use Theorem \ref{lowdimmain}.

For greater $n$ let $K \subset \R_x \times \R_y \times \R^{n-2}$. Assume the thesis is true for all cases with $n' < n$. If the theorem holds for strict $\Theta$ functions, then for any non-degenerate $\theta$ we take a sequence $\theta_i$ of strict $\Theta$ functions for $\eps = 1\slash i$. For any set $C$ for which $\theta(C)$ is defined, we have $\theta_i(C) \ra \theta(C)$, so as we had $\theta_i(A) \geq \theta_i(K_i) \geq \theta_i(\bA)$, we get the thesis when $i$ tends to infinity. Thus it is enough to restrict ourselves to strict assumptions. Note that under strict assumptions $\theta(U)$ is defined for any set $U$ with $\mi(U) > 0$ as $\supp \mi \subset \supp \eta_1$. In particular, if $\mi_2(\tU) > 0$, then $\theta(U \times \R^{n-2})$ is well defined.

Also note that if $\theta(\KO) = 0$, then $\eta_2$ has to be zero $\mi_2$-almost everywhere, which means $\theta(U) = 0$ for any $U$ such that it is defined, thus Theorem \ref{main} holds. Thus we can assume $\theta(\KO) > 0$.

We want to prove that for any $\eps > 0$ the quadrant $\KO$ is an $\eps$-appropriate set. 
We shall frequently require the following property from various sets $D$:
\begin{equation} \label{eqthetacon} \theta(D) = \theta(K),\end{equation}
or (for lower-dimensional sets)
\begin{equation} \label{eqthetaconpt} \theta_k(\mathbf{a};D) = \theta(K).
\end{equation}

We shall need to bound the diameter of the constructed sets from below. To this end consider the following sets: $\tS^0 = \{ (x,y) : \theta_{n-2}(x,y) = \theta(K)\}$, $\tS^+ = \{ (x,y) : \theta_{n-2}(x,y) \geq \theta(K)\}$, $\tS^- = \{ (x,y) : \theta_{n-2}(x,y) \leq \theta(K)\}$ and $\tS = \cl \tS^+ \cap \cl \tS^-$. We take a $\delta$-neighbourhood $\tS_\delta$ of $\tS$ with $\delta$ so small that $$\mi(\tS_\delta \setminus \tS) \leq \eps \mi_2(\tK) (\lambda_{n-2}(K \cap \{x = 0, y=0\}) \sup_K \eta_1)^{-1} \slash 3.$$ Note that as from property (\ref{t2}) the function $\theta_{n-2}(x,y)$ is coordinate-wise decreasing, the set $\tS \setminus \tS^0$ has measure 0.

\begin{rem}\label{sdelta}Note that any $\tD \subset \R_x \times \R_y$ having property (\ref{eqthetacon}) must, from Fact \ref{dividepoints}, have a non-empty intersection both with $\tS^+$ and $\tS^-$ in some points where the density of $\mi_2$ is positive. Thus any convex set $\tD$ with property \ref{eqthetacon} will satisfy $\tD \cap \supp \mi_2 \cap \tS\neq \emptyset$, as the set $\tD \cap \supp \mi_2$ is convex, and thus connected. Thus either $\tD \cap \supp \mi_2$ is contained in $\tS_\delta$ or it has diameter at least $\delta$.
\end{rem}

The main part of the proof will be an transfinite inductive construction of subsequent $\eps$-appropriate strict lens sets by the following Theorem:

\begin{thm}\label{divide}
Let $n > 2$. Assume Theorem \ref{main} holds under non-degenerate assumptions for any $n' < n$. Then under strict assumptions if $\theta(\KO) > 0$ for any ordinal $\gamma$ there exists a division of the set $\tKO$ into $\gamma + 2$ sets $U(\gamma,\beta)$ for $0 \leq \beta \leq \gamma+1$ satisfying:
\begin{itemize}
\item The set $U(\gamma,\gamma +1)$ is of $\mi_2$ measure at most $\eps' \mi_2(\tK) (\lambda_{n-2} (K \cap \{x = 0, y=0\}) \sup_K \eta_1)^{-1}$.
\item The set $U(\gamma,\gamma)$ is either an appropriate set, a strict lens set satisfying condition (\ref{eqthetacon}) or has $\mi_2$ measure 0. 
\item All sets $U(\gamma,\beta)$ for $\beta < \gamma$ are either $\eps'$-appropriate sets, empty, or have non-zero $\mi_2$ measure and satisfy $U(\gamma,\beta) \cap \tKO \subset \tS_\delta$
\item If any $U(\gamma,\beta)$ is empty for $\beta < \gamma$, then $U(\gamma,\gamma)$ has measure 0.\end{itemize} \end{thm}

If we prove this Theorem, we can apply it to prove Theorem \ref{main}. By the inductive assumption we assume Theorem \ref{main} holds for $n' < n$. We take $\gamma = \omega_1$ and $\eps' = \eps \slash 3$. As the measure of $\tKO$ is finite, it cannot have $\omega_1$ disjoint subsets of non-zero measure, thus some of $U(\omega_1,\beta)$ for $\beta < \omega_1$ are empty. Thus $U(\omega_1,\omega_1)$ has measure 0. 

Let $\tT$ be the sum of those $U(\omega_1, \beta)$ which are subsets of $\tS_\delta$. For every point $(x,y)$ in $\tT \cap \tS^0$ we apply Theorem \ref{main} to the restrictions of $K,A,\theta$ and $\mi$ to $(x,y) \times \R^{n-2}$. The conditions are satisfied --- the restiction of $\theta$ is a derivative of $\theta$ and thus non-degenerate, the restriction of $K$ is an generalized Orlicz ball due to Lemma \ref{sekcjajest} and the restriction of $\mi$ is a proper measure due to \ref{sekcjami}, the restriction of a c-set is obviously a c-set. Thus for all $(x,y) \in \tS^0$ we have $$\theta_{n-2}(x,y;A) \geq \theta_{n-2}(x,y) \geq \theta_{n-2} (x,y;\bA),$$
and as $\theta_{n-2}(x,y) = \theta(K)$ from the definition of $\tS^0$, from Lemma \ref{dividepoints} we get $$\theta \big(((\tT \cap \tS^0) \times \R^{n-2}) \cap A\big) \geq \theta(K) \geq 
\theta\big(((\tT \cap \tS^0) \times \R^{n-2}) \cap \bA\big),$$ and also $\theta((\tT \cap \tS^0) \times \R^{n-2}) = \theta(K)$, if only $\mi (\tT \cap \tS^0) > 0$. Thus $\tT \cap \tS^0$ either has measure 0, or is an appropriate set. Meanwhile $\tT \setminus \tS^0$ has measure at most $\eps (\lambda_{n-2} (K \cap \{x = 0, y=0\}) \sup_K \eta_1)^{-1} \mi_2(\tK) \slash 3$ from the definition of $\tS_\delta$.

We therefore have a division of $\tKO$ except a set of measure $2\eps (\lambda_{n-2} (K \cap \{x = 0, y=0\}) \sup_K \eta_1)^{-1} \mi_2(\tK) \slash 3$ into $(\eps \slash 3)$-appropriate sets. The sum of all the $(\eps \slash 3)$-appropriate sets is by Remark \ref{epssum} an $(\eps \slash 3)$-appropriate set. As the integral of $\eta_i$ over the remaining set is at most $2 \eps \mi_2(\tK) \slash 3$, the whole $\tKO$ is an $\eps$-appropriate set with the same $C_{U,i}$. As we can do this for any $\eps > 0$, by Lemma \ref{epstotot} $K$ is an appropriate set, which is the thesis of Theorem \ref{main}
\end{proof}

\subsection{Almost horizontal divisions} \label{aux3}
We shall prove that if we can divide a lens set with a horizontal, or even almost horizontal (under strict assumptions) line into two sets with equal $\theta$, then the lens set is appropriate.

\begin{lemma} \label{horline}
Assume Theorem \ref{main} holds for $n' < n$. Under strict assumptions if for a given lens set $D \subset \R^n$ satisfying (\ref{eqthetacon}) there exists a horizontal or vertical line $L$ in $\R_x \times \R_y$ dividing $\tD$ into two sets $\tD_-$ and $\tD_+$ of non-zero $\mi_2$-measure with $\theta(\tD_- \times \R^{n-2}) = \theta(\tD_+ \times \R^{n-2})$, then $D$ is an appropriate set.
\end{lemma}

\begin{proof} 
Suppose, without loss of generality, the line is horizontal given by $y = a_0$. From Corollary \ref{thetamain} we know that for any $a$ we also have $\theta_{n-1}(y = a;D) = \theta(D)$ if defined. 

From Lemma \ref{sekcjami} and property (\ref{t8}) we know that the restriction of $\mi$ to $\{y = a\}$ is a proper measure and the restriction of $\theta$ is a non-degenerate $\Theta$ function. From the assumption we can apply Theorem \ref{main}, thus $$\theta_{n-1}(y = a; D \cap A) \geq \theta_{n-1}(y = a; D) \geq \theta_{n-1}(y=a;D \cap \bA).$$ As $\theta_{n-1}(y=a;D) = \theta(D)$, which does not depend on $a$, we can apply Lemma \ref{dividesets} to get $\theta(D\cap A) \geq \theta(D) \geq \theta(D \cap \bA)$, and as $\theta(D) = \theta(K)$ this means $D$ is appropriate.
\end{proof}

\begin{lemma}\label{almhorline}
Assume Theorem \ref{main} holds for $n' < n$. Under strict assumptions if for a given strict lens set $D$ satisfying (\ref{eqthetacon}) for every $\beta > 0$ there exists a line $L_\beta$ with inclination between $0$ and $\beta$ (i.e. almost horizontal) or between $\frac{\pi}{2} -\beta$ and $\frac{\pi}{2}$ (i.e. almost vertical) dividing $\tD$ into two sets $\tD_-$ and $\tD_+$ of non-zero $\mi_2$-measure with $\theta(D_-) = \theta(D_+) = \theta(D)$, then $D$ is an appropriate set.
\end{lemma}

\begin{proof}
Assume $\theta(D) > 0$ (otherwise the thesis is trivial). We choose a sequence of such lines $L_i$ with $\beta \ra 0$. We choose a subsequence such that all lines are almost vertical or all are almost horizontal (we shall assume without loss of generality that all are almost horizontal). From the compactness of the set of lines intersecting the closure of $\tD \cap \supp(\mi_2)$ we can choose a subsequence of lines converging to some line $L$, which, of course, will be horizontal. If $L$ cuts off a non-zero $\mi_2$ measure both above and below it, then both the sets into which $\tD$ is divided have the same $\theta = \theta(D)$ from the continuity of $\theta$ with respect to the set and the thesis follows from Lemma \ref{horline}. The case left to examine is when $L_i$ approaches the lowest or highest point $p$ of $\tD \cap \supp(\mi_2)$.

From the definition of a lens set we know that the only points of $\tD$ on which $\mi_2$ vanishes lie outside $\tKO$. Thus the lowest point of $\tD \cap \supp(\mi_2)$ is the lower extremal point of $\tD$. The highest point can be either the upper extremal point of $\tD$ or can lie on the boundary of $\supp \mi_2$.

First consider the second, simpler case. As $D$ is a strict lens set and $K$ is proper, for any neighbourhood $\tU \subset \R_x \times \R_y$ of the highest point $p$ if we take a sufficiently horizontal line passing sufficiently close to $p$, the set it will cut off from $\tD$ will be a subset of $\tU$ ($\tD \cap \supp \mi_2$ has no horizontal edges). We know that $\supp \eta_2 \subset \Int\ \supp \mi$, so $\cl\ \widetilde{\supp \eta_2} \subset \supp \mi_2$. As $p$ lies on the boundary of $\supp \mi_2$, it lies outside $\Int\ \supp \mi_2$ and thus outside $\cl\ \widetilde{\supp \eta_2}$, so we can choose an open neighbourhood $\tU$ of $p$ on which $\eta_2$ is 0. This neighbourhood has non-zero $\mi_2$ measure, and as $\mi_2(\tD) > 0$, $\mi_2(\tD \cap \tU) > 0$. But $\eta_2$ on the whole set $\tU$ is zero, thus any line cutting off only a part of $\tU$ cannot satisfy $\theta(D_+) = \theta(D) > 0$.

In the first case $(L_i)$ approaches one of the extremal points of $\tD$. Assume it is the lower point. For any line $L_i$ the set $\tD^{L_i}_-$ is a lens set. From Lemma \ref{dividepoints} there has to be a point $p_i \in D^{L_i}_-$ with $\theta(\{p_i\} \times \R^{n-2}) \leq \theta(D^{L_i}_-) = \theta(D)$. The lines $L_i$ tend to the horizontal line through $(x_1,y_1)$, the lower extremal point of $D$. Thus, the vertical coordinate of $p_i$ tends to $y_1$, and as $\tD$ has no horizontal edges, the horizontal coordinate of $p_i$ tends to $x_1$, meaning $p_i \ra (x_1,y_1)$.

From property (\ref{t2}), as $\theta_{n-2}(\{p_i\} \times \R^{n-2}) \leq \theta(D)$, for all points $p \in D$ except for $(x_1,y_1)$ we have $\theta_{n-2}(\{p\} \times \R^{n-2}) \leq \theta(D)$. This, however, from Lemma \ref{dividepoints} implies in particular, that for any horizontal line $M$ dividing $\tD$ into two sets of non-zero $\mi$-measure we have $\theta(D_+) \leq \theta(D)$, which from Lemma \ref{thetamain} implies $\theta(D_+) = \theta(D)$, which from Lemma \ref{horline} implies that $D$ is appropriate.
\end{proof}

\subsection{$\eps$-appropriateness of lens sets}
This subsection puts down precisely what we meant by ``long and narrow'' in the idea of the proof, and show how to go from the ``longness and narrowness'' to $\eps$-appropriateness.

\begin{lemma} \label{convexproj}
Let $C \subset \R^n$ be a convex set with $\lambda(C) > 0$, let $I \subset C$ be an interval of length $a$, and let $L$ be the line containing $I$. Let $f : C \ra (0,\infty)$ be a $1\slash m$-concave function. Let $P : \R^m \ra L$ be the orthogonal projection onto $L$. Let $J \subset I$ be an subinterval of length $b$. Let $C' \subset C$ be such a set that $P(C') \subset J$. Then $$\int_{C'} f(x) dx \leq \Big(\Big(\frac{a+b}{a-b}\Big)^{n+m} - 1\Big) \int_C f(x) dx$$ and also $$\int_{C'} f(x) dx \leq \frac{2^{n+m+2} b}{a} \int_{C} f(x) dx.$$
\end{lemma}

\begin{proof} Let $p(y) = \int_{x : P(x) = y} f(x) dx$ and let $I' = \{y \in L : p(y) > 0\}$. From Fact \ref{convconc} the function $p(y) = \int_{x : P(x) = y} f(x) dx$ is a $(1 \slash m+n-1)$-concave function on $L$, thus $I'$ is an interval. As $f$ is positive and $C$ is convex and has positive measure, $p$ is positive on $\Int I$, thus the length of $I'$ is at least $a$. If $J \cap I' = \emptyset$, then $\int_{C'} f(x) dx = 0$ and the thesis is satisfied, so assume $J \cap I' \neq \emptyset$. Then $I' \setminus J$ is a sum of two intervals (one may be empty) of total length at least $a-b$. Thus it contains an interval $I''$ of length at least $\frac{a-b}{2}$, let $\{y_1\} = \cl I''\cap \cl J$ and $y_2$ be the other end of $I''$. Let $y_2$ and $y_3$ be the ends of $I'$ and let $T$ be such that $y_3 = Ty_1 + (1-T)y_2$ (as $y_1$ lies between $y_2$ and $y_3$ we know $T \geq 1$).

As $p$ is $1\slash (n+m-1)$-concave, $$p^{1\slash (n+m-1)}(ty_1 + (1-t)y_2) \geq tp^{1\slash (n+m-1)}(y_1) + (1-t)p^{1\slash (n+m-1)}(y_2) \geq tp^{1\slash (n+m-1)}(y_1)$$ for $t \in [0,1]$, which means $$\int_{I''} p(y)dy  \geq |I''| \int_{[0,1]} t^{n+m-1} p(y_1)dt = |I''| \frac{1}{n+m} p(y_1).$$
Similarly for $T \geq t \geq 1$ we have $$p^{1\slash (n+m-1)}(ty_1 + (1-t)y_2) \leq tp^{1\slash (n+m-1)}(y_1) + (1-t)p^{1\slash (n+m-1)}(y_2) \leq tp^{1\slash (n+m-1)}(y_1)$$ for $t \in [1,T]$, which gives $$\int_J p(y) \leq |I''| \int_1^{(|J| + |I''|)\slash |I''|} t^{n+m-1} p(y_1) = |I''| \frac{1}{n+m}\Big(\Big(\frac{a+b}{a-b}\Big)^{n+m} - 1\Big) p(y_1).$$ Thus $$\frac{\int_{C'} f(x) dx}{\int_C f(x) dx} \leq \Big(\frac{a+b}{a-b}\Big)^{n+m} - 1,$$ which proves the first part of the Lemma.

For the second part note that if $a \slash b \leq 2^{m+n+2}$, then the thesis is true, as $\int_{C'} f(x) dx \leq \int_C f(x) dx$ because $C' \subset C$. For $b \slash a \leq 2^{-(n+m+1)}$ we have $$\Big(\frac{a+b}{a-b}\Big)^{n+m} - 1 = \Big(\frac{1+b \slash a}{1-b \slash a}\Big)^{n+m} - 1 \leq \frac{1 + 2^{n+m} b\slash a}{1 - 2^{n+m} b\slash a} - 1 \leq \frac{2^{n+m+1} b\slash a}{1 \slash 2}  = 2^{n+m+2} b \slash a.$$
\end{proof}

\begin{cor}\label{2dim} Let $\eps > 0$. Let $\mi$ be a measure on $\R^2$ with a $1\slash m$ concave density. Let $\tD \subset \R^2$ be a lens set. Let $L$ be the extremal line of $\tD$ and $p : \R^2 \ra L$ the orthogonal projection onto $L$. Let $A$ be a c-set in $\R^2$.  Let $A' = p^{-1}(A \cap L)$. Assume the relevant length of $\tD$ (that is, the length of $L \cap \tD \cap \supp \mi$) is at least $d > 0$, the inclination of $\tD$ between $\beta$ and $\pi\slash 2 - \beta$ with $\beta > 0$ and width at most $w = \frac{1}{2 \max(\cot \beta ,\tan \beta)} 2^{-m-3} \eps d $. Then $\mi((A \bigtriangleup A') \cap \tD) \leq \eps \mi(\tD)$.\end{cor}

% (!!! TU POJAWIA SIE POJECIE RELEVANT COSTAM --- TRZEBA USPOJNIC, PEWNIE ZAARGUMENTOWAC, ZE JAK REL DIAM JEST DUZY, A SZEROKOSC MALA, A $\supp\mi$ JEST C-SETEM, TO DLUGOSC WZDLUZ EXTREMAL LINE TEZ DUZA)

\begin{proof}
Let $p=(x_p,y_p)$ be the rightmost point on $L \cap A$ (and at the same time on $L \cap A'$, from the definition of $A'$). As both $A$ and $A'$ are c-sets, we have $A \bigtriangleup A' \subset \{(x,y) : x > x_p, y < y_p\} \cup \{(x,y) : x < x_p, y > y_p\}$. As $D$ has width at most $w$, the projection of $(A \bigtriangleup A') \cap D$ onto $L$ has length at most $2w \max(\tan \beta , \cot \beta)$. From Lemma \ref{convexproj} we know that as $2w \max(\tan \beta , \cot \beta) < 2^{-m-3} \eps d$, we have $\mi((A \bigtriangleup A') \cap D) \leq \eps \mi(D)$.
\end{proof}

\begin{cor} \label{ndim} Let $K \subset \R_x \times \R_y \times \R^{n-2}$ be a generalized Orlicz ball with a proper measure $\mi$ and $D$ be a lens set of relevant length at least $d$, inclination between $\beta$ and $\pi \slash 2 - \beta$ and width at most $w = \frac{1}{2 \max(\cot\beta , \tan\beta)} 2^{-m-3} \eps d (\lambda_{n-2}(K \cap \{x = 0,y=0\}))^{-1}$, where $m$ is such that the density of $\mi$ is $1\slash m$ concave. Let $A$ be a c-set in $\R^n$ and let $A'$ be defined as before. Then $$\mi((A \bigtriangleup A') \cap D) \leq \eps \mi_2 (\tD).$$
\end{cor}

\begin{proof}
For each $t \in \R^{n-2}$ we may apply Corollary \ref{2dim}, and integrate over $K \cap \{x = 0, y = 0\}$ to get a bound for the Lebesgue measure.
\end{proof}

Note the same argument works if $A$ is the complement of a c-set.

\begin{cor} \label{fundim} Let $K \subset \R_x \times \R_y \times \R^{n-2}$ be a generalized Orlicz ball with a proper measure $\mi$. Let $D$ be a lens set of relevant length at least $d$, inclination between $\beta$ and $\pi \slash 2 - \beta$ and width at most $w = \frac{1}{2 \max (\cot\beta , \tan\beta)} 2^{-m-3} d M^{-1} \eps (\lambda_{n-2}(K \cap \{x = 0,y=0\}))^{-1}$ and $\phi: \R^n \ra [0,M]$ be a coordinate-wise decreasing function with $\bar{\phi}(t) := \phi(p(t))$, where $p$ is the orthogonal projection onto $L \times \R^{n-2}$, $L$ being the extremal line of $\tD$.  Then for any $U \subset D$ we have $$\big|\int_U \phi(t) d\mi(t) - \int_U \bar{\phi}(t) d\mi(t)\big| < \eps \mi_2 (\tD).$$
\end{cor}

\begin{proof}
As $\phi$ is coordinate-wise decreasing, the sets $\phi^{-1} ([s,\infty))$ are c-sets. By the integration by parts, $$\int_U \phi(t) d\mi(t) = \int_0^M \mi(\phi^{-1}([s,\infty) \cap U) ds.$$ The sets $(\bar{\phi})^{-1}([s,\infty))$ are formed from the sets $\phi^{-1}([s,\infty))$ as in Corollary \ref{ndim}. Thus for each $s$ we have 
\begin{align*} 
\Bigg|\mi\Big(\phi^{-1}([s,\infty) \cap U\Big) - \mi\Big((\bar{\phi})^{-1}([s,\infty) \cap U\Big)\Bigg| & \leq \mi\Big(\big(\phi^{-1}([s,\infty) \cap U\big) \bigtriangleup \big((\bar{\phi})^{-1}([s,\infty) \cap U\big)\Big)  \\ & \leq \mi\Big(\big(\phi^{-1}([s,\infty) \cap D\big) \bigtriangleup \big((\bar{\phi})^{-1}([s,\infty) \cap D\big)\Big) \leq M^{-1} \eps \mi_2(\tD),\end{align*} which integrated over $[0,M]$ gives the thesis.
\end{proof}

\begin{lemma} \label{lensapp}
Consider a generalized Orlicz ball $K \subset \R_x \times \R_y \times \R^{n-2}$ with a proper measure $\mi$ with both its defining functions $1\slash m$ concave, a strict non-degenerate $\Theta$ function, any $\eps > 0$ and any c-set $A$. Assume Theorem \ref{main} holds for $n' < n$. Let $D$ be a lens set satisfying $\theta(D) = \theta(K)$ of relevant length at least $\delta$, inclination between $\beta$ and $\frac{\pi}{2} - \beta$ and width at most $$w = \frac{1}{2 \max (\cot\beta, \tan\beta)} 2^{-3m-4} d \min\{1,(\sup_K \eta_1)^{-1}\} \eps (\lambda_{n-2}(K \cap \{x = 0,y=0\}))^{-1}.$$ Then $D$ is an $8\eps$-appropriate set.
\end{lemma}

\begin{proof} 
Let $L$ be the extremal line of $D$. We switch coordinates in the plane $\R_x \times \R_y$ to orthogonal coordinates $(u,v)$ such that $L = \{v = 0\}$ and $u > 0$ on the positive quadrant of $\R_x \times \R_y$. Define for any set $U \in \{A,\bA\}$ the set $U'$ by $U \cap \{v = 0\}$ and $U''$ by $U \times \R_v$. Let $K'$ be a generalized Orlicz ball in $\R_u \times \R^{n-2}$ such that $\KO' = \KO \cap \{v = 0\}$ given by Lemma \ref{sekcjajest} and $K'' = K' \times \R_v$. Let $\eta_i'$ be the restriction of $\eta_i$ to $\{v = 0\}$ and $\eta_i''(u,v,t) = \eta_i(u,0,t)$. Let $\mi'$ be the measure on $K'$ with density $h(u) = \int_{\R_v} \1_{(u,v) \in \tD} f(u,v) g(u,v)$, where $f$ and $g$ are the density functions defining $\mi$, and $\mi''$ be the measure on $\R^n$ with density $f(x)g(y)$ (without restricting to $K$).

We want to prove that $\int_{U' \cap K'} \eta_i' d\mi'$ is a good approximation of $\int_{U \cap D} \eta_i d\mi$, then check the assumptions for Theorem \ref{main} on $K'$ and apply it for $A'$. First note that $$\int_{K'} \phi(u,t) d\mi'(u,t) = \int_{K'' \cap D} \phi(u,0,t) d\mi''(u,v,t)$$ for any function $\phi$ defined on $K'$. This follows directly from the definitions of $K''$, $\mi'$ and $\mi''$.

Let $M = \min\{1,(\sup_K \eta_1)^{-1}\}$. We know $\KO$ is a c-set, thus $\mi((K \bigtriangleup K'') \cap D) \leq \eps \mi_2(\tD)$ by Corollary \ref{ndim}. Thus for any $\phi$ we have 
\begin{align*} \Bigg|\int_{K'} \phi(u,t) d\mi'(u,t) & - \int_{K \cap D} \phi(u,0,t) d\mi(u,v,t)\Bigg|  = \\ &\Bigg|\int_{K'' \cap D} \phi(u,t) d\mi''(u,v,t) - \int_{K \cap D} \phi(u,0,t) d\mi''(u,v,t)\Bigg| < M \eps \mi_2(\tD) \sup|\phi|.\end{align*}
We repeat the same trick for $U \in \{A'', \bA''\}$, putting $\phi' = \phi \cdot \1_U$ in the above inequality and applying Corollary \ref{ndim} again to get 
\begin{align*} \Bigg|\int_{K' \cap A'} \phi(u,t) d\mi'(u,t) - \int_{K \cap D \cap A} \phi(u,0,t) d\mi(u,v,t) \Bigg| < M \eps \mi_2(\tD) \sup |\phi|.\end{align*}
 Finally, we insert $\eta_i$ for $\phi$ and apply Corollary \ref{fundim} to get
$$\Bigg|\int_{K' \cap A'} \eta_i'(u,t) d\mi'(u,t) - \int_{K \cap D \cap A} \eta_i(u,v,t) d\mi(u,v,t) \Bigg| \leq 3 \eps \mi_2(\tD),$$ and the same for integration over $K \cap D \cap \bA$ and $K \cap D$.

Now we want to check assumptions for Theorem \ref{main}. $K'$ is a generalized Orlicz ball due to Lemma \ref{sekcjajest}. $A'$ is a c-set in $\R_u \times \R^{n-2}$ because $L$ is positively inclined, thus an increase in $u$ translates to an increase in both $x$ and $y$. $\mi'$ is a projection of the measure with the density $f(x) g(y) \1_D$. The first two functions are $1\slash m$ concave, the third is $1 \slash 1$ concave as $D$ is convex. Thus from Facts \ref{prodncon} and \ref{convconc} the density $h(u)$ of $\mi'$ is a $1\slash (3m + 1)$ concave function. Thus $\mi'$ is a proper measure on $K'$ (recall $\mi'$ is restricted to $K'$, thus the third point of the Definition \ref{propermeasure} is satisfied). $\eta_1'$ and $\eta_2'$ are restrictions of $\eta_1$ and $\eta_2$ to $K'$, which is a derivative of $K$, thus they define a non-degenerate $\Theta$-function on $K'$.

Let us apply Theorem \ref{main}. We get \begin{equation}\frac{\int_{K' \cap A'} \eta_2'(u,t) d\mi'(u,t)}{\int_{K' \cap A'} \eta_1'(u,t) d\mi'(u,t)} \geq \frac{\int_{K'} \eta_2'(u,t) d\mi'(u,t)}{\int_{K'} \eta_1'(u,t) d\mi'(u,t)} \geq \frac{\int_{K' \cap \bA'} \eta_2'(u,t) d\mi'(u,t)}{\int_{K' \cap \bA'} \eta_1'(u,t) d\mi'(u,t)}.\label{4listopada}\end{equation} We need to make the middle expression equal to $\theta(D)$, so for any $u_0, t_0$ we define $$\bar{\eta}_i'(u_0,t_0) = \frac{\int_{K \cap D} \eta_i(u,v,t) d\mi(u,v,t)}{\int_{K'} \eta_i'(u,t) d\mi'(u,t)} \eta_i'(u_0,t_0).$$ As $\bar{\eta}_i' = C_i \eta_i'$, we have inequalities (\ref{4listopada}) for functions $\bar{\eta}_i'$ (although they do not necessarily define a $\Theta$ function on $K'$). To bound the error we have 
\begin{align*} \int_{K'} \big|\bar{\eta}_i'(u,t) - \eta_i'(u,t)\big|d\mi'(u,t) &= \int_{K'} \eta_i'(u,t) \Big| \frac{\int_{K \cap D} \eta_i(u,v,t) d\mi(u,v,t)}{\int_{K'} \eta_i'(u,t) d\mi'(u,t)} - 1 \Big| d\mi'(u,t) \\ &= \Bigg|\int_{K \cap D} \eta_i(u,v,t) d\mi(u,v,t) - \int_{K'} \eta_i'(u,t) d\mi'(u,t)\Bigg| \leq 3 \eps \mi_2(\tD).\end{align*} As we bounded the integral of errors, the error on $K' \cap A'$ and $K' \cap \bA'$ is no larger than $3 \eps \mi_2(\tD)$.

We can now for $U \in \{A',\bA'\}$ and $i \in \{1,2\}$ put $C_{U,i} = \int_{K' \cap U} \bar{\eta}_i' d\mi'$. Applying inequalities (\ref{4listopada}) to $\bar{\eta}_i'$ we get
$$\frac{C_{A,2}}{C_{A,1}} \geq \frac{\int_{K'} {\bar{\eta}}_2'(u,t) d\mi'(u,t)}{\int_{K'} \bar{\eta}_1'(u,t) d\mi'(u,t)} = \frac{\int_{K\cap D} \eta_2(u,v,t) d\mi(u,v,t)}{\int_{K\cap D} \eta_1(u,v,t) d\mi(u,v,t)} = \theta(D) = \theta(K) \geq \frac{C_{\bA,2}}{C_{\bA,1}},$$ and putting together all the estimates we made we get $|C_{U,i} - \int_{K \cap D \cap U} \eta_i d\mi| \leq 6 \eps \mi_2(\tD)$.
\end{proof}

\section{The transfinite induction}
What is left to prove is Theorem \ref{divide}. We will prove by transfinite induction an extended version of Theorem \ref{divide}, which will allow us to carry the information we need through the induction steps. The sets $U(\gamma, \beta)$ will have to satisfy the conditions of Theorem \ref{divide}, and furthermore the following conditions:

\begin{itemize}
\item For any $\gamma > \beta$ we have $U(\gamma,\beta) = U(\beta+1,\beta)$.
\item For any $\gamma$ we have $U(\gamma,\gamma+1) = U(0,1)$.
\item If $\gamma$ is a successor ordinal and $U(\gamma,\gamma)$ has positive $\mi_2$ measure, the sets $U(\gamma,\gamma-1)$ and $U(\gamma,\gamma)$ are formed by dividing $U(\gamma-1,\gamma-1)$ with a straight line of positive inclination. 
\item If $\gamma$ is a limit ordinal, $U(\gamma,\gamma) = \bigcap_{\beta < \gamma} U(\beta,\beta)$. 
\item For any $\gamma$ if $U(\gamma,\gamma)$ has positive $\mi_2$ measure, then for all $\beta < \gamma$ the sets $U(\beta,\beta)$ are strict lens sets. 
\end{itemize}
Remark that this in fact means we carry out a transfinite inductive construction. The sets $U(\gamma,\beta)$ for $\beta < \gamma$ depend only on the second argument, once constructed. The set $U(\gamma,\gamma+1)$ is equal to $U(0,1)$. The set $U(\gamma,\gamma)$ in each step has a part cut off to make a new set $U(\gamma+1,\gamma+1)$.

Note that if $\theta(K) = 0$, then $K$ is appropriate (as any $U \subset K$ with $\mi_2(U) > 0$ has $\theta(U) = 0$). Thus by putting $U(\gamma,0) = \tKO$ for any $\gamma$ and $U(\gamma,\beta) = \emptyset$ for $\gamma + 1 \geq \beta > 0$ we satisfy the conditions of Theorem \ref{divide}. Thus, further on, we assume $\theta(K) > 0$.

\subsection{Starting the transfinite induction}
First we need to define the sets $U(0,0)$ and $U(0,1)$ to start the induction. If we take $D = [x_-,x_+] \times [y_-,y_+] \times \R^{n-2}$, then $D$ is a lens set and satisfies condition (\ref{eqthetacon}). It is not, however, a strict lens set. 

The idea is to take two almost vertical lines --- one close to the left edge of $\tD$ and the other close to the right edge, then look at the $\theta$ of the set they cut off. If $\theta$ is too large, we move the left line closer to the edge, if too small, we move the right line closer to the edge. When we have balanced $\theta$, we repeat the same for horizontal lines. 
By cutting off a bit from each edge we shall also ensure $[x_1,x_2] \subset (x_-,x_+)$ and similarly for $y$.
Below is a formalization of the argument.

If $\tKO$ is appropriate to begin with, we take $U(0,1) = \emptyset$ and $U(0,0) = \tKO$. Thus we assume $\tKO$ is not appropriate.

Denote by $L^-(x,\beta)$ the line through $(x,y_-)$ with inclination $\pi\slash 2 - \beta$ and by $L^+(x,\beta)$ the line through $(x,y_+)$ with inclination $\pi\slash 2 - \beta$. Denote by $\tD^-(x,\beta)$ the subset of $[x_-,x_+]\times [y_-\times y_+]$ to the left of $L^-(x,\beta)$ and by $\tD^+(x,\beta)$ the subset to the right of $L^+(x,\beta)$. Note that for $\beta \in (0,\pi \slash 2)$ those sets have positive $\mi_2$ measure by the definition of a proper measure. Let $\phi^-(x,\beta) = \theta(\tD^-(x,\beta) \times \R^{n-2}) - \theta(K)$ and $\phi^+(x,\beta) = \theta(\tD^+(x,\beta) \times \R^{n-2})-\theta(K)$. From property \ref{t6} these functions are continuous in both arguments.
From Lemma \ref{thetamain} and Lemma \ref{almhorline} there is a $\beta_0 > 0$ such that for $\beta < \beta_0$ we have $\phi^-(x,\beta) > 0$ and $\phi^+(x,\beta) < 0$ for $x \in (x_-,x_+)$.

Now start with any $x_l$, $x_u$ and $0 < \beta_l, \beta_u < \beta_0$ such that the sets $\tD^-(x_l,\beta_l)$ and $\tD^+(x_u,\beta_l)$ have measure no larger than
$\eps' (\lambda_{n-2} (K \cap \{x = 0, y=0\}) \sup_K \eta_1)^{-1} \mi(K) \slash 4$ and do not intersect. Now if we fix $x_u$ and $\beta_u$ while letting $x_l$ tend to $x_-$ and $\beta_l$ to 0, then $\theta$ of the sum of the two sets will tend to $\theta(\tD^+(x_u,\beta_u)\times \R^{n-2})$, which is strictly smaller than $\theta(K)$. If, on the other hand, we fix $x_l$ and $\beta_l$ and let $x_u$ tend to $x_+$ and $\beta_u$ to 0, the $\theta$ of the two sets will approach $\theta(\tD^-(x_l,\beta_l)\times \R^{n-2})$, which is strictly greater than $\theta(K)$. Thus, from the Darboux property, for some values  $x_- < x_l <  x_u < x_+$ and $\beta_l$ and $\beta_u$ we have the function $$\theta \Big(\tD^+(x_u,\beta_u)\times \R^{n-2} \cup \tD^-(x_l,\beta_l)\times \R^{n-2}\Big) = \theta(K).$$

The set that remains is a lens set with no vertical boundaries and satisfies property (\ref{eqthetacon}). If it is appropriate, we have found our $U(0,0)$ and define $U(0,1) = \tD^-(x_l,\beta_l) \cup \tD^+(x_u,\beta_u)$. If not, then we can repeat the same trick for $y$ (we needed the non-appropriateness to use Lemma \ref{horline}), and achieve a lens set with no horizontal and no vertical boundaries and separated from $x_-$ and $x_+$, i.e. a strict lens set.

Thus we define $U(0,1) = \tD^-(x_l,\beta_l) \cup \tD^+(x_u,\beta_u) \cup \tD^-(y_l,\alpha_l) \cup \tD^+(y_u,\alpha_u)$ and and $U(0,0) = ([x_-,x_+] \times [y_-,y_+]) \setminus U(0,1)$.

\begin{rem} \label{boundden} Assume $U(0,0)$ is a strict lens set (otherwise the induction will be trivial). Recall $f$ and $g$ are $1\slash m$-concave functions defining the proper measure $\mi$. As $U(0,0)$ is a strict lens set, it is separated from the boundary of the support of $f \cdot g$. Thus (as $f$ and $g$ are continuous on the interior of their support), they both attain positive minimal values $f_L$ and $g_L$. Also, as they are continuous on their support and $1\slash m$ concave, they are bounded from above by some $f_U$ and $g_U$. Thus for any set $T \subset U(0,0)$ we have $$f_U g_U \lambda_2(T) \geq \mi_2(T) \geq f_L g_L \lambda_2(T),$$ and for any function $t$ on $T$ we have $$f_U g_U \int_T t(p) d\lambda_2(p) \geq \int_T t(p) d\mi_2(p) \geq f_L g_L \int_T t(p) d\mi_2(p).$$\end{rem}

\subsection{The induction step for successor ordinals}
For a successor ordinal $\gamma + 1$ we have a division of $\tKO$ for $\gamma$. We put $U(\gamma+1,\gamma+2) = U(\gamma,\gamma+1)$. If $U(\gamma,\gamma)$ is appropriate of positive measure, we put $U(\gamma+1,\gamma) = U(\gamma,\gamma)$ (as an appropriate set is an $\eps$-appropriate set) and $U(\gamma+1,\gamma+1) = \emptyset$. If $U(\gamma,\gamma)$ has measure 0, we put $U(\gamma+1,\gamma) = \emptyset$ and $U(\gamma+1,\gamma+1) = U(\gamma,\gamma)$. The difficult case to deal with will be when $U(\gamma,\gamma)$ is a non-appropriate strict lens set. For brevity denote $U(\gamma,\gamma)$ by $\tD$.

In this case from Lemma \ref{almhorline} there exists an angle $\alpha' > 0$ such that any positively inclinated line dividing $\tD$ into two sets of non-zero $\mi$-measure with equal $\theta$ has inclination greater than $\alpha'$ and smaller than $\frac{\pi}{2} - \alpha'$. If the inclination of $\tD$ is $\alpha''$, let $\alpha = \min\{\alpha',\alpha'',\frac{\pi}{2} - \alpha''\}$.

We shall attempt to cut off a ``long and narrow'' lens set $U(\gamma+1,\gamma)$ from $U(\gamma,\gamma)$. We shall cut off a narrow set satisfying (\ref{eqthetacon}). From Remark \ref{sdelta} it will either be long, or be a subset of $\tS_\delta$, both of which satisfy us.

Take a sufficiently small $w$ ($w < \frac{1}{2 \max (\cot\alpha , \tan\alpha)} 2^{-3m-4} \delta \min\{1,(\sup_K \eta_1)^{-1}\} \frac{\eps}{8} (\lambda_{n-2}(K \cap \{x = 0,y=0\})^{-1}$, where $m$ is such that the density functions of $\mi$ are $1\slash m$-concave, will suffice). For any angle $\xi \in [0,\frac{\pi}{2}]$ we can find continuously a line $L_\xi$ of inclination $\xi$ such that the part $\tD_+(\xi)$ of $\tD \cap \supp \mi$ lying above and to the left of $L_\xi$ has width no larger than $w$. From Lemma \ref{thetamain} we have $\theta(D_+(0)) \geq \theta(D)$ and $\theta(D_+(\pi\slash2)) \leq \theta(D)$. From the Darboux property for some $\xi$ we have $\theta(D_+(\xi)) = \theta(D)$. We take $U(\gamma+1,\gamma) = \tD_+(\xi)$. Let $I_\xi$ denote the segment of $L_\xi$ intersecting $\tD$.

The set $U(\gamma + 1, \gamma + 1) = U(\gamma,\gamma) \setminus U(\gamma+1,\gamma)$ is, of course, a strict lens set, satisfying condition (\ref{eqthetacon}), because the new edge has inclination between $\alpha$ and $\frac{\pi}2 - \alpha$, and all the other edges come from the old set $\tD$. It remains to check that $U(\gamma+1,\gamma)$ satisfies the conditions of the transfinite induction. First let us check what is the inclination of $U(\gamma+1,\gamma)$. If both the ends $I_\xi$ fall upon the upper-left border of $U(\gamma,\gamma)$, then they are the extremal points of $U(\gamma+1,\gamma)$, and thus the inclination of $U(\gamma+1,\gamma)$ is the inclination of the segment, which is between $\alpha'$ and $\frac{\pi}{2} - \alpha'$. If one of them falls upon the lower-right border, then the extremal points of $U(\gamma+1,\gamma)$ are the end of $I_\xi$ on the upper-left border and one of the extremal points of $U(\gamma,\gamma)$, and the inclination of $U(\gamma+1,\gamma)$ is between the inclination of $U(\gamma+1,\gamma)$ and the inclination of the segment, which means it is between $\alpha$ and $\frac{\pi}{2} - \alpha$. If both ends fall upon the lower-right border, the extremal points of $U(\gamma+1,\gamma)$ are the extremal points of $U(\gamma,\gamma)$, which means $U(\gamma+1,\gamma)$ has inclination $\alpha''$.
Thus, the inclination of $U(\gamma,\gamma)$ is between $\alpha$ and $\frac{\pi}{2} - \alpha$.

If $U(\gamma+1,\gamma) \subset \tS_\delta$, the induction thesis is satisfied. Thus we may assume $U(\gamma+1,\gamma)$ sticks outside $\tS_\delta$. Note that as $\theta_{n-2}(p)$, $p \in \R_x \times \R_y$, is a coordinate-wise increasing function from property (\ref{t2}), one of the extremal points of $U(\gamma+1,\gamma)$ has to lie outside $\tS_\delta$, and at least one point of $\tS$ lies on the extremal line of $U(\gamma+1,\gamma)$. Thus, the length of the segment of the extremal line contained in $\tKO$ is at least $\delta$. 

Thus $U(\gamma+1,\gamma)$ has relevant length at least $\delta$, width at most $w$ and inclination between $\alpha$ and $\frac{\pi}{2} - \alpha$. Thus from Lemma \ref{lensapp} we know that $U(\gamma+1,\gamma)$ is $\eps$-appropriate, which means we completed the induction step.

\subsection{The induction step for limit ordinals}
For limit ordinals $\gamma$ the set $U(\gamma,\gamma + 1) = U(0,1)$, the sets $U(\gamma,\beta)$ for $\beta < \gamma$ are defined by $U(\gamma,\beta) = U(\beta+1,\beta)$, and from the inductive assumption the conditions for $U(\gamma,\beta)$ are met. We define $U(\gamma,\gamma)$ as the intersection $\bigcap_{\beta < \gamma} U(\beta,\beta)$. 

We have to check that $U(\gamma,\gamma)$ thus defined satisfies the induction thesis. 
If any of the sets $U(\gamma',\beta), \beta < \gamma'$ was empty, then from the inductive assumption $U(\gamma'+1,\gamma'+1)$ has $\mi_2$ measure 0 and thus $U(\gamma,\gamma)$ has $\mi_2$ measure 0, which satisfies the conditions. If $U(\beta,\beta)$ was not a strict lens set for some $\beta < \gamma$, then $U(\gamma,\gamma)$ has measure 0, again satisfying the conditions. The case to worry about is when $U(\gamma,\gamma)$ is a intersection of a descending family of strict lens sets satisfying condition (\ref{eqthetacon}) and has a positive $\mi_2$ measure.

A descending intersection of lens sets is a lens set --- the circumscribed rectangle is the intersection of circumscribed rectangles, the extremal points belong to the intersection, and the intersection is convex. A descending intersection of sets satisfying (\ref{eqthetacon}) with positive $\mi_2$ measure satisfies (\ref{eqthetacon}) by property (\ref{t6}). We have to prove that the intersection is either a strict lens set, or appropriate. As $U(\gamma,\gamma) \subset U(0,0)$, it is separated from $x_-,x_+,y_-$ and $y_+$. Thus we only have to check it does not have a horizontal or vertical edge.

Suppose $U(\gamma,\gamma)$ has a horizontal or vertical edge $I$. We may assume, without loss of generality, that $I$ is a horizontal edge. 
We will assume it is an upper horizontal edge. In the case of the lower one, the proof goes very similarily: every construction of new points
is done centrally-symetric, and every inequality is opposite. In one place, where the proof significantly changes, we will say it explicitly.

Let $(x_0, y_0)$ be the left end of $I$ and $(x_1, y_0)$ the right end. First we shall prove the following Lemma:

\begin{lemma} With the notation as previously we have $\cl I \cap \supp \eta_2 \neq \emptyset$. \label{5.9} \end{lemma}

\begin{proof}
We shall prove the Lemma by contradiction. Suppose that $\cl I \cap \supp \eta_2 = \emptyset$. The idea of the proof is that at some moment, a line dividing some $U(\beta,\beta)$ into $U(\beta+1,\beta+1)$ and $U(\beta+1,\beta)$ lies above $I$ and cuts off only points that are above and to the right of the left end of $I$, or almost so, and thus only cuts off points, which do not belong to $\supp \eta_2$. Thus $\eta_2$ is zero on the set $U(\beta+1,\beta)$ which was cut off, $\theta(U(\beta+1,\beta)) = 0$, a contradiction. Now for a formal proof:

As $\theta(U(\gamma,\gamma)) > 0$, some point of $U(\gamma,\gamma)$ has to lie inside $\supp \eta_2$, thus (as $\supp \eta_2$ is a c-set), the lower left extremal point of $U(\gamma,\gamma)$ lies in $\supp \eta_2$. Note, that as $\eta_2 = 0$ on $I$, $I$ has to be an upper edge, the lower edge case is trivial here. Let $x_2 < x_0$ be such that $(x_2,y_0) \not\in \supp \eta_2$. Then let $y_2 < y_0$ be a number so close to $y_0$ that $(x_2,y_2) \not\in \supp \eta_2$ and $(x_2,y_2) \not\in U(\gamma,\gamma)$. Take a $\beta < \gamma$ such that $(x_2,y_2) \not\in U(\beta,\beta)$. As $U(\beta,\beta)$ is a lens set, no points $(x_2,y)$ with $y > y_2$ belong to $U(\beta,\beta)$.

As $U(\beta,\beta)$ is a strict lens set, and $I \subset U(\beta,\beta)$, there exists a $y_3 > y_0$ such that $(x_1,y_3) \in U(\beta,\beta)$. Take $y_3$ to be so small that 
\begin{equation}\label{nachyl}
\frac{y_3 - y_0}{x_1-x_0} < \frac{y_0 - y_2}{x_0 - x_2}.
\end{equation}
Let $\beta'$ be the smallest such ordinal that $(x_1,y_3) \not\in U(\beta',\beta')$. Of course $\beta'> \beta$ and from the inductive assumption $\beta'$ is a successor ordinal. 

Let $L$ be the line which divides $U(\beta'-1,\beta'-1)$ into $U(\beta',\beta')$ and $U(\beta',\beta'-1)$. $L$ intersects the interval $[(x_1,y_0),(x_1,y_3)]$ and does not intersect $I$, so, from (\ref{nachyl}), $L$ intersects the line $x = x_2$ at some point above $(x_2,y_2)$. $U(\beta',\beta'-1) \subset U(\beta'-1,\beta'-1) \subset U(\beta,\beta)$, thus $U(\beta',\beta'-1)$ contains no points $(x_2,y)$ with $y > y_2$. Thus all points from $U(\beta',\beta'-1)$ lie above and to the right of $(x_2,y_2)$. As $\supp \eta_2$ is a c-set and $(x_2,y_2) \not\in \supp \eta_2$, we have $U(\beta',\beta'-1) \cap \supp \eta_2 = \emptyset$, thus $\theta(U(\beta',\beta'-1)) = 0$. But as we assumed $\theta(K) > 0$ this means that $U(\beta',\beta'-1)$ is empty, a contradiction.
\end{proof}

Thus we know that $\cl I \cap \supp \eta_2 \neq \emptyset$, and as $\Int\ \supp \eta_1 \supset \supp \eta_2$, there is an interval $I' \subset I \cap \tKO$ of positive length, which means $\theta_{n-1}(y_0;I\times \R^{n-2})$ is defined. The idea of the proof in this case is to prove that $\theta_{n-1}(y_0;I\times\R^{n-2}) = \theta(K)$, which from Lemma \ref{thetamain} and Lemma \ref{horline} will imply $U(\gamma,\gamma)$ is appropriate. We prove this by selecting a moment at which the set $U(\beta+1,\beta)$ which is being cut off lies above $I$, and comparing its $\theta$ (which we know to be $\theta(K)$) to $\theta_{n-1}(x_0; I\times\R^{n-2})$. The formal proof goes as follows:

We assume $I$ is an upper horizontal edge. In the case of $I$ being a lower horizontal edge, the below construction works centrally-symetrically.Recall $(x_0, y_0)$ be the left end of $I$ and $(x_1,y_0)$ the right end. Take any $0 < \eps < |I|$. Take $x_2 = x_0 - \eps$ and $y_2 < y_0$ and close enough that $(x_2,y_2) \not\in U(\gamma,\gamma)$. Take $\beta_1 < \gamma$ such that $(x_2,y_2) \not\in U(\beta_1,\beta_1)$. Next take a point $(x_1,y_3)$ with $y_3 > y_0$ such that (\ref{nachyl}) is satisfied, and take $\gamma > \beta_2 > \beta_1$ such that the upper extremal point of $U(\beta_2,\beta_2)$ lies below $y_3$. Again, as in the proof of Lemma \ref{5.9}, any line dividing some $U(\beta,\beta)$ for $\beta > \beta_2$ and crossing $x = x_1$ between $y_3$ and $y_0$ will exit $U(\beta,\beta)$ at some $x > x_2$. For $\gamma > \beta > \beta_2$ any line cutting off the upper extremal point $p$ of $U(\beta,\beta)$ will cross $x = x_1$ between $y_3$ and $y_0$ because $p$ will lie below $y_3$ (as $U(\beta,\beta) \subset U(\beta_2,\beta_2)$ and to the right of and above $(x_1,y_0)$ as $U(\beta,\beta) \subset U(\gamma,\gamma)$ and the line has to go below $p$ and above $(x_1,y_0)$ as $(x_1,y_0) \in U(\beta,\beta)$.

Let us consider the functions $\teta_i(x,y) = \int_{\R^{n-2}} \eta_i(x,y,t) dt$ for $i = 1,2$. The set $[x_-,x_+] \times [y_-,y_+]$ is compact and $\teta_i$ are continuous from property (\ref{t4}) (recall $n > 2$), thus we can find a $\tilde{\delta} > 0$ such that $$\|p_1 - p_2\| < \tilde{\delta}\ \ \Ra\ \ |\teta_i(p_1) - \teta_i(p_2)| < \eps$$ for $i = 1,2$. Also, as $g$ (the density of $\mi$ with respect to $y$) is $1\slash m$-concave, it is continuous on the interior of its support, and thus we can take $\tilde{\delta}$ such that also $|g(p_1) - g(p_2)| < \eps$.

If $\teta_1((x_1,y_0)) > 0$ take $\delta = \tilde{\delta}$. If not, then as $\Int \tKO \supset \supp \teta_2$, there exists an interval $J' \subset I \cap (\Int \tKO \setminus \supp \teta_2)$ of positive length $c$. As $\R_x \times \R_y \setminus \Int \tKO$ and $\supp \teta_2$ are closed, we may take $\delta \leq \tilde{\delta}$ small enough, that there exists an interval $J \subset I$ of length at least $c\slash 2$, such that $$J \times [y_0-\delta,y_0 + \delta] \subset \Int \tKO \setminus \supp \teta_2.$$

Take $\gamma > \beta_3 > \beta_2$ such that the whole set $U(\beta_3,\beta_3)$ lies below the line $y = y_0 + \delta$.

Now let $(x_4,y_4)$ be the upper right extremal point of $U(\beta_3,\beta_3)$. Let $\beta_4$ be the first $\beta$ such that $(x_4,y_4) \not\in U(\beta_4,\beta_4)$. The ordinal $\beta_4$ has to be a successor, let $L'$ be the line dividing $U(\beta_4 - 1, \beta_4 -1)$ into $U(\beta_4,\beta_4 - 1)$ and $U(\beta_4,\beta_4)$, and let $l$ be the inclination of $L'$. Any tangent to the upper-left border of $U(\beta_4,\beta_4)$ has inclination no smaller than $l$. Let $\beta_5$ be the first ordinal greater than $\beta_4$ for which some tangent to the upper left edge of $U(\beta_5, \beta_5)$ has inlination strictly smaller than $l$. Again, $\beta_5$ has to be a successor ordinal. Let $L$ be the line dividing $U(\beta_5 - 1,\beta_5 - 1)$ into $U(\beta_5,\beta_5 - 1)$ and $U(\beta_5,\beta_5)$. This line has to go above $I$, to become a part of the upper edge of $U(\beta_5,\beta_5)$. As the inclination of this line is smaller than the inclination of any tangent to the upper left edge of $U(\beta_5 - 1,\beta_5 - 1)$, the right end of $L \cap U(\beta_5 - 1,\beta_5 - 1)$ lies on the lower right edge of $U(\beta_5 - 1,\beta_5 - 1)$. It lies above $y_0$, as it goes above $I$ and has positive inclination, and lies to the right of $x_1$, as the lower right edge of $u(\beta_5 - 1,\beta_5 - 1)$ above $y_0$ lies to the right of $x_1$.

Now we will prove some inequalities on $\theta$. In the case of $I$ being lower edge, the inequalities are simply reversed. Let $\tD = \tD(\eps)$ be the part of $U(\beta_5,\beta_5 - 1)$ that lies to the left of $x = x_1$. As usual, $D = D(\eps) = \tD(\eps) \times \R^{n-2}$. As $U(\beta_5, \beta_5 - 1)$ is a lens set, from Lemma \ref{thetamain} we know $$\theta(\tD \times \R^{n-2}) \leq \theta(U(\beta_5,\beta_5 - 1) \times \R^{n-2}) = \theta(K).$$ Remark that the line $L''$ that cut $(\beta_5,\beta_5 - 1)$ off contains the whole lower edge of $\tD$. Thus as the inclination of $L''$ is smaller than the inclination of the upper edge of $\tD$ the function $x \mapsto \lambda_1 (\tD_x)$, where $\tD_x$ is the section of $\tD$ at $x$, is strictly increasing.

If $\tD$ has $\mi_2$ measure 0, then the lower extremal point $p_5$ of $U(\beta_5,\beta_5 - 1)$ lies above and to the right of any point of $U(\gamma,\gamma)$. However, from property (\ref{t2}) $$\theta_{n-2} (p_5) \leq \theta(U(\beta_5,\beta_5 - 1) \times \R^{n-2}) = \theta(K),$$ which means that from property (\ref{t2}) for any point $p \in U(\gamma,\gamma)$ we have $$\theta_{n-2}(p) \leq \theta_{n-2}(p_5) \leq \theta(K).$$ However, we know $\theta(U(\gamma,\gamma) \times \R^{n-2}) = \theta(K)$, which, from Fact \ref{dividesets} implies that for almost all points in $U(\gamma,\gamma)$ we have $\theta_{n-2}(p) = \theta(K)$. Thus any horizontal line divides $U(\gamma,\gamma)$ into two sets with equal $\theta$, which from Lemma \ref{horline} implies $U(\gamma,\gamma)$ is appropriate. Hereafter we shall assume $\mi_2(\tD) > 0$.

Note that the whole set $\tD$ lies in the rectangle $[x_2,x_1] \times [y_0 - \delta,y_0 + \delta]$. It lies to the left of $x_1$ from its definition. To the right of $x_2$ as $\beta_5 > \beta_2$. Below $y_0 + \delta$ because $\beta_5 > \beta_3$. Above $y_0 - \delta$ because its lower edge is the line $L''$ which passes above $(x_0,y_0)$, so if it dipped below $y_0 - \delta$, it would also (as $\eps < |I|$) have to reach above $y_0 + \delta$.

Now we want to estimate $\theta_{n-1}(y_0;I\times \R^{n-2})$ by $\theta(\tD \times \R^{n-2})$. This will, unfortunately, involve quite a lot of technicalities. We begin with a lemma:

\begin{lemma} \label{compla} There exist two numbers $c_1, c_2 > 0$ independent of $\eps$ such that for sufficiently small $\eps > 0$ and a set $\tD$ constructed as above for this $\eps$ we have 
$$\lambda_2(\tD \cap \{(x,y) : \teta_i(x,y) > c_1\}) > c_2 \lambda_2(\tD),$$ for $i = 1,2$.\end{lemma}

\begin{proof}
The proof for this lemma is a bit different for $I$ being a lower edge. First, let us prove it for an upper edge.

First we prove the thesis for $\teta_1$. Suppose $\teta_1(x_1,y_0) > 0$. Then supposing $\eps < \frac{1}{2} \teta_1(x_1,y_0)$ for any $(x,y) \in \tD$ we have $$\teta_1(x,y) \geq \teta_1(x_1,y) \geq \teta_1(x_1,y_0) - \eps > \frac{1}{2}\teta_1(x_1,y_0),$$ as $|y - y_0| < \delta$ and $\teta_1$ is decreasing as $\eta_1$ is decreasing,
thus it is enough to have $c_1 < \frac{1}{2} \teta_1(x_1,y_0)$ and $c_2 < 1$.

In the case $\teta_1(x_1,y_0) = 0$ let $b_1 = \sup \{x : \teta_1(x,y_0) > 0\}$. Recall that we constructed an interval $J$ of length $c$ (independent of $\eps$) such that $J \times [y_0 - \delta, y_0 + \delta] \subset \supp \teta_1 \setminus \supp \teta_2$. Let $J = [j_0,j_1]$. Now as $\tD \subset [x_2,x_1] \times [y_0 - \delta,y_0 + \delta]$ for $x \in J$ and $(x,y) \in \tD$ we have $\teta_2(x,y) = 0$ and $\teta_1(x,y) > 0$, which means $j_2 \leq b_1$. On the other hand $\theta(\tD) \geq \theta(K) > 0$, thus $\tD$ contains points with positive $\eta_2$, and thus for these points $(x,y)$ we have $x < j_0$. Note that as $\lambda_1(\tD_x)$ is strictly increasing, so if $\tD$ condains some point to the left of $j_0$, then for every $x \in J$ the set $\tD_x$ has positive Lebesgue measure.

Let $j = \frac{j_0 + j_1}{2}$ be the midpoint of $J$. If $\eps < \frac{1}{2} \teta_1(j,y_0)$ we have 
\begin{align*} \lambda_2\Bigg(\tD \cap \Big\{(x,y) : \teta_1(x,y) > \frac{1}{2}\teta_1(j, y_0)\Big\}\Bigg) &\geq
\lambda_2 \Bigg(\Big\{(x,y) \in \tD : \teta_1(x,y_0) \geq \teta_1(j,y_0)\Big\}\Bigg) \\ &\geq \lambda_2\Big(\big\{(x,y) \in \tD : x < j\big\}\Big).\end{align*}

Now we perform a similar operation as in Lemma \ref{convexproj}. The function $p(x) = \lambda(\tD_x)$ is concave on its support, $p(j_0) \geq 0$, thus for every $t \in [0,1]$ we have $p((1-t)j_0 + tj) \geq t p(j)$ and for $t > 1$ we have $p((1-t)j_0 + tj) \leq t p(j)$. Thus 
$$\lambda_2 \Big(\{(x,y) \in \tD : x < j\}\Big) = \int_{x < j} p(x) \geq |j - j_0| \int_0^1 t p(j) = \frac{j - j_0}{2} p(j).$$
In a similar vein
$$\lambda_2 \Big(\{x,y) \in \tD : x \geq j\}\Big) = \int_{x \geq j} p(x) \leq |j - j_0| \int_1^{\frac{x_1 - j_0}{j - j_0}} t p(j) = \frac{j - j_0}{2} \Bigg(\frac{(x_1 - j_0)^2}{(j-j_0)^2} - 1\Bigg) p(j),$$
which gives us:
$$\frac{\lambda_2(\tD)}{\lambda_2(\tD \cap \{(x,y) : \teta_i(x,y) 
> \frac{1}{2}\teta_1(j,y_0)\})} \leq 1 + \frac{\lambda_2 (\{x,y) \in \tD 
: x \geq j\})}{\lambda_2 (\{x,y) \in \tD : x < j\})} \leq
1 + \frac{(x_1 - j_0)^2}{(j-j_0)^2} - 1 = \frac{(x_1 - j_0)^2}{(j-j_0)^2},$$
which gives the thesis for $c_1 \leq \frac{1}{2}\teta_1(j,y_0)$ and $c_2 \leq \frac{(j-j_0)^2}{(x_1 - j_0)^2}$. 

To deal with $\teta_2$ first use Remark \ref{boundden} to get
$$\int_\tD \teta_2(x,y) d\mi_2(x,y) = \theta(\tD \times \R^{n-2}) \int_{\tD} \teta_1(x,y) d\mi_2(x,y) \geq \theta(K) c_1 c_2 f_L g_L \lambda_2(\tD).$$
On the other hand $\teta_2$ is bounded from above on $\supp \teta_2$ by $M = \teta_2(0,0)$, as it is continuous. We have 
\begin{align*}
f_L g_L c_1 c_2 \theta(K) \lambda_2(\tD) 
&\leq \int_D \teta_2(x,y) d\mi_2(x,y) 
\leq f_U g_U \int_D \teta_2(x,y) d\lambda_2
\\ &\leq f_U g_U \big(M \lambda_2(\tD \cap \{ \teta_2(x,y) > a\}) + a \lambda(\tD)\big).
\end{align*}
The above holds for any $a$. Let us take $2a =  \frac{c_1 c_2 \theta(K) f_L g_L}{f_U g_U}$. Then we have 
$$\lambda_2(\tD \cap \{\teta_2(x,y) > a \}) \geq a \lambda_2(\tD) \slash M,$$
which implies (with the assumption $\eps < a \slash 2$) 
$$\lambda(\tD \cap \{\teta_2(x,y_0) > a \slash 2\}) \geq \lambda(\tD \cap \{\teta_2(x,y_0) > a - \eps\}) \geq (a \slash M) \lambda(D).$$

Now, let us assume that $I$ is a lower horizontal edge. The proof is much easier in that case. Since $\theta(U(\gamma, \gamma)) > 0$, 
  there is a segment $I' \subset I$ starting at lower left end of $I$, such that $I' \subset \supp \teta_2$. Moreover, we can take
  such $I'' \subset I'$, that on $I''$ we have $\teta_2 > c$ for some $c$. Since $x \to \lambda(\tD_x)$ is decreasing on $I$, 
  we have 
  $$\lambda_2(\tD \cap \{(x,y): \teta(x,y)>c \}) \geq \lambda_2(\tD \cap I'' \times \R) \geq \lambda_2(\tD)\frac{|I''|}{|I|}.$$
\end{proof}

\begin{cor} \label{calkateta}There exists a constant $c_3$ such that for all sufficiently small $\eps$ we have $$\int_\tD \teta_i(x,y) d\mi_2(x,y) \geq c_3 \mi_2(\tD).$$\end{cor}
\begin{proof}
$$\int_\tD \teta_i(x,y) d\mi_2 \geq 
\int_\tD \teta_i(x,y) \1_{\teta_i(x,y) > c_1} d\mi_2 \geq c_1 c_2 \lambda_2(\tD) \geq 
c_1 c_2 f_L g_L\mi_2(\tD).$$
\end{proof}

The rest of the proof is independent of the fact, whether $I$ is lower or upper edge, we simply use already proven facts.

Now to estimate $\theta(D(\eps))$. As $\beta_5 > \beta_2$ we know $\|(x,y) - (x,y_0)\| < \delta$, thus $|\teta_i(x,y) - \teta_i(x,y_0)| < \eps$. Thus we get:
\begin{align*} 
\theta(D(\eps)) &= 
\frac{\int_{\tD(\eps)} \teta_2(x,y) d\mi_2(x,y)}{\int_{\tD(\eps)} \teta_1(x,y) d\mi_2(x,y)} \leq 
\frac{\int_{\tD(\eps)} \teta_2(x,y_0) + \eps d\mi_2(x,y)}{\int_{\tD(\eps)} \teta_1(x,y_0) - \eps d\mi_2(x,y)}  \\
&= \frac{\int_{\tD(\eps)} \teta_2(x,y_0) + \eps d\mi_2(x,y)}{\int_{\tD(\eps)} \teta_2(x,y_0) d\mi_2(x,y)} \cdot
\frac{\int_{\tD(\eps)} \teta_1(x,y_0) d\mi_2(x,y)}{\int_{\tD(\eps)} \teta_1(x,y_0) - \eps d\mi_2(x,y)} \cdot
\frac{\int_{\tD(\eps)} \teta_2(x,y_0) d\mi_2(x,y)}{\int_{\tD(\eps)} \teta_1(x,y_0) d\mi_2(x,y)}.\end{align*}

The first and second fraction will both be bounded by 1 as $\eps \ra 0$ from Corollary \ref{calkateta}:

\begin{align*}
\frac{\int_{\tD(\eps)} \teta_2(x,y_0) + \eps\ d\mi_2(x,y)}{\int_{\tD(\eps)} \teta_2(x,y_0) d\mi_2(x,y)} - 1 = 
\frac{\eps \int_{\tD(\eps)} d\mi_2(x,y)}{\int_{\tD(\eps)} \teta_2(x,y_0) d\mi_2(x,y)} \leq \frac{\eps \mi_2(\tD(\eps))}{\int_{\tD(\eps)} \teta_2(x,y) - \eps\ d\mi_2(x,y)} = \frac{\eps}{c_3 - \eps},\end{align*} 
and (here we prove that the lower bound for the reciprocal converges to 1, which is equivalent)
\begin{align*}
\frac{\int_{\tD(\eps)} \teta_1(x,y_0) - \eps\ d\mi_2(x,y)}{\int_{\tD(\eps)} \teta_1(x,y_0) d\mi_2(x,y)} - 1 = 
\frac{-\eps \int_{\tD(\eps)} d\mi_2(x,y)}{\int_{\tD(\eps)} \teta_1(x,y_0) d\mi_2(x,y)} \geq \frac{-\eps \mi_2(\tD(\eps))}{\int_{\tD(\eps)} \teta_1(x,y) - \eps\ d\mi_2(x,y)} = \frac{-\eps}{c_3 - \eps}.\end{align*}
The third fraction is the one that should converge to (or at least, for very small $\eps$, be bounded by) $\theta_{n-1}(y_0;I\times \R^{n-2})$. Let $I_\eps = [x_0 - \eps,x_1] = [x_2,x_1]$. As $\|(x,y) - (x,y_0)\| < \delta$, we have:
\begin{align*}
\frac{\int_{\tD(\eps)} \teta_2(x,y_0) d\mi_2(x,y)}{\int_{\tD(\eps)} \teta_1(x,y_0) d\mi_2(x,y)} =
\frac{\int_{I_\eps} \int_{\tD_x(\eps)} \teta_2(x,y_0) f(x) g(y) dx dy)}{\int_{I_\eps} \int_{\tD_x(\eps)} \teta_1(x,y_0) f(x) g(y) dx dy)} \leq \frac{g(y_0) + \eps}{g(y_0) - \eps} \cdot \frac{\int_{I_\eps} \teta_1(x,y_0) f(x) \lambda(\tD_x) dx}{\int_{I_\eps} \teta_2(x,y_0) f(x) \lambda(\tD_x) dx}.\end{align*}
The first of these fractions obviously tends to $1$ as $g(y_0) \geq g_L > 0$. The second can be bounded using Lemma \ref{gencheb}, part 3:
\begin{align*} 
\frac{\int_{I_\eps} \teta_1(x,y_0) f(x) \lambda(\tD_x) dx}{\int_{I_\eps} \teta_2(x,y_0) f(x) \lambda(\tD_x) dx} \leq 
\frac{\int_{I_\eps} \teta_1(x,y_0) f(x) dx} {\int_{I_\eps} \teta_2(x,y_0) f(x) dx} = \frac{\int_{I_\eps} \teta_1(x,y_0) f(x) g(y_0) dx} {\int_{I_\eps} \teta_2(x,y_0) f(x) g(y_0) dx} = \theta_{n-1}(y_0;I_\eps \times \R^{n-2})
\end{align*}
From property (\ref{t6}) used for restrictions to $y = y_0$ we have $\theta_{n-1}(y_0;I_\eps \times \R^{n-2}) \ra \theta_{n-1}(y_0;I\times \R^{n-2})$ when $\eps \ra 0$.

Putting all the estimates together we get $\theta(K) \leq \theta(D(\eps)) \leq c(\eps) \theta_{n-1}(y_0;I \times \R^{n-2})$, where $c(\eps) \ra 1$. Thus we can go with $\eps$ to 0 to get $\theta(K) \leq \theta_{n-1}(y_0;I\times \R^{n-2})$. On the other hand from Lemma \ref{thetamain} we have $\theta_{n-1}(y_0;I\times \R^{n-2}) \leq \theta(U(\gamma,\gamma)) = \theta(K)$, which means $\theta_{n-1}(y_0;I\times \R^{n-2}) = \theta(K)$. From Lemma \ref{incthetay} this means that for any horizontal line $L$ intersecting $U(\gamma,\gamma)$ we have $\theta(I) \leq \theta(U(\gamma,\gamma) \cap L)$, which, from Lemma \ref{thetamain} implies that any horizontal line divides $U(\gamma,\gamma)$ into two sets with equal $\theta$. Thus, from Lemma \ref{horline}, $U(\gamma,\gamma)$ is appropriate.

This finishes the proof of the inductive step in the limit ordinal case: the assumption $U(\gamma,\gamma)$ has positive measure and is not a strict lens set led us to the conclusion it is appropriate.

\section{$\Theta$ functions on Orlicz balls}
Our main target is proving Theorem \ref{orliczglownetw}:

Due to Lemma \ref{prelim} we need to prove inequality (\ref{wzorek_mi2}) for any c-sets $A \subset \R^k$ and $B\subset \R^{n-k}$. We shall attempt to prove (\ref{wzorek_mi2}) using Theorem \ref{main}.  

\subsection{The one-dimensional case --- the $\phi$ functions}
First we need to apply the Brunn-Minkowski theorem to get a $\Theta$-like condition:

\begin{lemma}\label{BMlemma}
Let $K \subset \R_x \times \R_y \times \R^{n-2}$ be a generalized Orlicz ball. Let $0 \leq x_1 \leq x_2 \in \R_x$, $0 \leq y_1 \leq y_2 \in \R_y$. Let $K_{x_i,y_j} = K \cap (\{(x_i,y_j)\} \times \R^{n-2})$ for $i,j \in \{1,2\}$. Let $\nu$ be a log-concave measure on $\R^{n-2}$. Then
$$\nu(K_{x_1,y_1}) \cdot \nu(K_{x_2,y_2}) \leq \nu(K_{x_1,y_2}) \cdot \nu(K_{x_2,y_1}).$$
\end{lemma}

\begin{proof}
Let $f_i$, $i = 1,2,\ldots,n$ be the Young functions of $K$, with $f_1$ defined on $\R_x$ and $f_2$ on $\R_y$. Let us consider the generalized Orlicz ball $K' \in \R^{n-1}$, with the Young functions $\Phi_i = f_{i+1}$ for $i > 1$ and $\Phi_1(t) = t$ --- that is, we replace the first two functions with a single identity function.

For any $x \in \R$ let $P_x$ denote the set $K' \cap (\{x\} \times \R^{n-2})$, and $|P_x| = \nu(P_x)$. As $K'$ is a convex set, from the Brunn-Minkowski inequality (see for instance \cite{ff}) the function $x \mapsto |P_x|$ is a log-concave function, which means that for any $t \in [0,1]$ we have 
$$ |P_{tx + (1-t)y}| \geq |P_x|^t|P_y|^{1-t}.$$

In particular, for given real non-negative numbers $a,b,c$ we have
$$|P_{a+c}| \geq |P_a|^{b \slash (b+c)} |P_{a+b+c}|^{c \slash (b+c)},$$
$$|P_{a+b}| \geq |P_a|^{c \slash (b+c)} |P_{a+b+c}|^{b \slash (b+c)},$$
and as a consequence when we multiply the two inequalities,
\begin{equation} \label{PABC} |P_{a+b}| \ |P_{a+c}| \geq |P_a| \ |P_{a+b+c}|.\end{equation}

Now let us take $a = f_1(x_1) + f_2(y_1)$, $b = f_1(x_2) - f_1(x_1)$ and $c = f_2(y_2) - f_2(y_1)$. As the Young functions are non-negative and increasing on $[0,\infty)$, the numbers $a,b,c$ are non-negative. From the definitions above we have:
$$K_{x_1,y_1} = \{(z_3,\ldots,z_n) \in \R^{n-2} : f_1(x_1) + f_2(y_1) + \sum_{i=3}^n f_i(z_i) \leq 1\}
= \{(z_i)_{i=3}^n : \Phi_1(a) + \sum_{i=3}^n \Phi_{i-1}(z_i) \leq 1\} = P_a.$$
Similarily we have $K_{x_2,y_1} = P_{a+b}$, $K_{x_1,y_2} = P_{a+c}$ and $K_{x_2,y_2} = P_{a+b+c}$. Substituting those values into inequality (\ref{PABC}) we get the thesis.
\end{proof}

First we consider $K \subset \R^{n-1} \times \R_z$. Take any $z_2 > z_1 > 0$ and consider any c-set $B$ in $\R^{n-1}$. We define $\phi_1(x) = \1_K(x,z_1)$ and $\phi_2(x) = \1_K(x,z_2)$ for $x \in \R^{n-1}$. Let $\KO' = (\KO)_{z = z_1}$. By Lemma \ref{sekcjajest} $\KO'$ is a positive quadrant of some generalized Orlicz ball $K'$.

\begin{lemma}\label{metaderivative} If $\bar{K}'$ is a derivative of $K'$, then there exists a generalized Orlicz ball $\bar{K}$ such that $\phi_j(x)$ on $\bar{K}'$ is equal to $\1_{\bar{K}}(x,z_j)$ for $j \in \{1,2\}$.\end{lemma}

\begin{proof} We have a sequence $K' = K_0', K_1',\ldots,K_m' = \bar{K}'$ where $K_{i+1}'$ is some restriction of $K_i'$. We can, taking identical restrictions (that is, restrictions to hyperplanes defined by the same equations or to the same intervals with respect to the same variables), construct a sequence $K = K_0,K_1,\ldots,K_m = \bar{K}$ such that $K_i' = (K_i)_{z = z_1}$. As $z$ was not a variable of $\R^{n-1}$ of which $K'$ was a subset, on each step being a hyperplane restriction $z$ does not appear in the equation of the restriction hyperplane, thus we can speak of a $z$ variable in all $K_i$, and the isometric immersion $u : \bar{K} \hookrightarrow K$ maps $(\bar{K})_{z = z_j}$ into $K_{z = z_j}$. Thus $\1_{\bar{K}}(x,z_j) = \1_{K}(u(x,z_j))$, which (when, as always, we identify $\bar{K}$ with its image in $K$) gives the thesis.
\end{proof}

\begin{lemma}\label{phithetat4} For any generalized Orlicz ball $K \subset \R^{m-1} \times \R_z$, any $z_2 > z_1 > 0$, any coordinate-wise decomposition $\R^{m-1} = \R^k \times \R^{m-k-1}$ and any proper measure $\mi$ on $K' = K_{z=z_1}$ the function $$\theta^1_k(y) = \frac{\int_{\R^k} \1_K(x,y,z_2) d\mi_{|\R^k}(x)}{\int_{\R^k} \1_K(x,y,z_1) d\mi_{|\R^k}(x)}$$ is coordinate-wise decreasing on $\R^{m-k-1}$.\end{lemma}

\begin{proof}
Let $l = m-k-1$. Select any coordinate variable $y_i$ from $\R^l$ and fix all other variables $\mathbf{y}$ in $\R^l$ at some $\mathbf{y}_0$. For $y_1 \leq y_2$ we have to prove 
$$\frac{\int_{\R^k} \1_K(x,\mathbf{y}_0,y_1,z_2) d\mi_{|\R^k}(x)}{\int_{\R^k} \1_K(x,\mathbf{y}_0,y_1,z_1) d\mi_{|\R^k}(x)} \geq \frac{\int_{\R^k} \1_K(x,\mathbf{y}_0,y_2,z_2) d\mi_{|\R^k}(x)}{\int_{\R^k} \1_K(x,\mathbf{y}_0,y_2,z_1) d\mi_{|\R^k}(x)}.$$ 
The intersection $K_{\mathbf{y} = \mathbf{y}_0}$ is a generalized Orlicz ball from Lemma \ref{sekcjajest} and the restriction of $\mi$ is a proper measure from Lemma \ref{sekcjami}. Thus taking $K'' = K_{\mathbf{y} = \mathbf{y}_0}$ we have to prove
$$\frac{\int_{\R^k} \1_{K''}(x,y_1,z_2) d\mi_{|\R^k}(x)}{\int_{\R^k} \1_{K''}(x,y_1,z_1) d\mi_{|\R^k}(x)} \geq \frac{\int_{\R^k} \1_{K''}(x,y_2,z_2) d\mi_{|\R^k}(x)}{\int_{\R^k} \1_{K''}(x,y_2,z_1) d\mi_{|\R^k}(x)}.$$ 
Note that even if the density of $\mi$ changes with $y$, it cancels out in both fractions, thus we can assume the density of $\mi$ changes only on $\R^k$. As a proper measure has a $1\slash m$-concave density, and thus a log-concave density, we can apply Lemma \ref{BMlemma} to get the thesis.
\end{proof}

\begin{lemma}\label{phitheta} The functions $\phi_1$ and $\phi_2$ defined as above define a $\Theta$ function on $K'$. \end{lemma}

\begin{proof}
We have to check the four properties defining $\Theta$ functions. Property (\ref{t-1}) is obvious, both $\phi_1$ and $\phi_2$ are bounded by one. Note that $K$ is a c-set, as it is convex and 1-symmetric, which immediately gives properties (\ref{t0}) and (\ref{t1}).

Condition (\ref{t2}) is a consequence of Lemma \ref{phithetat4}. If $\bar{K}'$ is any derivative of $K'$, then from Lemma \ref{metaderivative} we have some $\bar{K}$ such that $\phi_j$ restricted to $\bar{K}'$ are equal to $\1_{\bar{K}}(\cdot,z_j)$, and thus from Lemma \ref{phithetat4} the appropriate ratio of integrals is coordinate-wise decreasing.
\end{proof}

\begin{lemma} \label{phipthetas} If $K$ is a proper generalized Orlicz ball, then $\phi_1$ and $\phi_2$ define a strict $\Theta$ function. \end{lemma}

\begin{proof}
The properties (\ref{s3}) and (\ref{t4}) are trivial. For property (\ref{s2}) notice that as the Young functions are strictly increasing, $\Int K_{z = z_1} \supset K_{z=z_2}$. 

To check property (\ref{s4}) we have to prove that $\int_{\R^k} \1_K(x,y,z_j) d\mi_{|\R^k}(x) = \mi_{|\R^k}(K_{y,z_j})$ is continuous in $y$ for $j = 1,2$ and $k > 0$. Let $\mi_k$ denote $\mi_{|\R^k}$. Let us take any sequence $y^i \ra y^\infty$. First note that as the Young functions $f_l$ do not assume the value $+\infty$, they are continuous. Thus $\sum f_l(y_l^i) \ra \sum f_l(y_l^i)$.

Let $L_a = \{x \in \R^k : \sum f_i(x_i) \leq 1 - a\}$, let $a_l = \sum f_l(y_l^i) + f_z(z_j)$ and $a = \sum f_l(y_l) + f_z(z_j)$. We know $a_l \ra a$, we want to prove $\mi_k (L_{a_l}) \ra \mi_k(L_a)$. However, 
\begin{align*} 
\lim_{l\ra \infty} \mi_k(L_{a_l}) \leq \lim_{t \ra 0^+} \mi_k(L_{a + t}) = \mi_k(\bigcap_{t > 0} L_{a+t}) = \mi_k(L_a)
\end{align*}
as measure is continuous with respect to the set, and
\begin{align*} 
\lim_{l\ra \infty} \mi_k(L_{a_l}) \geq \lim_{t \ra 0^-} \mi_k(L_{a + t}) = \mi_k(\bigcap_{t < 0} L_{a+t}) = \mi_k(L_a),
\end{align*}
where we use the fact that $\mi_k(\{x \in \R^k : \sum f_i(x_i) = 1 - a\}) = 0$, as $f_i$ are strictly increasing. Thus $\mi_k(K_{y_l,z_j}) \ra \mi_k(L_{y,z_j})$, which proves property (\ref{s4}).
\end{proof}

\begin{cor} \label{phinondeg} For any generalized Orlicz ball $K$ the functions $\phi_1$ and $\phi_2$ define a non-degenerate $\Theta$ function.\end{cor}

\begin{proof}
First we prove that $\phi_1$ and $\phi_2$ define a weakly non-degenerate $\Theta$ function. From Lemma \ref{approrlicz} we can approximate $K$ with a proper generalized Orlicz ball $K'$ satisfying $K' \subset K$ and $\lambda(K \setminus K') < \eps \slash 2$. Additionally, from Corollary \ref{apprsection} we may take $z_1'$ and $z_2'$ such that $K' \cap \{z = z_j'\}$ approximates $K \cap \{z = z_j\}$ up to a set of $\lambda$ measure $\eps$. 

We take $\phi_1'(x) = \1_{K'}(x,z_1')$ and $\phi_2'(x) = \1_{K'}(x,z_2')$. As the intersections of $K'$ at $z_j'$ were good approximations of intersections of $K$ at $z_i$, we have $\int |\phi_i - \phi_i'| d\lambda = \lambda(K_{z=z_1} \bigtriangleup K'_{z=z_1'}) \leq \eps$. From Lemma \ref{approrlicz} we know $K'$ is a proper generalized Orlicz ball and $K' \subset K$. From Lemma \ref{phipthetas} we know that $\phi_1'$ and $\phi_2'$ define a strict $\Theta$ function. Thus $\phi_1$ and $\phi_2$ define a weakly non-degenerate $\Theta$ function.

As for the derivatives of the function defined by $\phi_1$ and $\phi_2$ by Lemma \ref{metaderivative} they are constructed in the same manner on some derivative of $K$, and thus also define a weakly non-degenerate $\Theta$ function. Thus $\phi_1$ and $\phi_2$ define a non-degenerate $\Theta$ function.
\end{proof}

\begin{cor} \label{wniosekjedno}
For any generalized Orlicz ball $K \subset \R^n$ and any c-set $A \subset \R^{n-1}$ the function $$z \mapsto \frac{\int_\bA \1_K(z,x) d\mi(x)}{\int_{\R^{n-1}} \1_K(z,x) d\mi(x)}$$ is a decreasing function of $z$ where defined.\end{cor}

\begin{proof}
From Corollary \ref{phinondeg} we can apply Theorem \ref{main} to the $\Theta$ function defined by $\phi_1$, $\phi_2$ to get for any $0 \leq z_1 < z_2$:
\begin{equation} \label{pndg1} \frac{\int_A \1_K(x,z_2) d\mi(x)}{\int_A \1_K(x,z_1) d\mi(x)} \geq \frac{\int_\bA \1_K(x,z_2) d\mi(x)}{\int_\bA \1_K(x,z_1) d\mi(x)},\end{equation} if both sides are defined.
We can apply Fact \ref{obvi} to make it
\begin{equation}\label{pndg2} \frac{\int_{\R^{n-1}} \1_K(x,z_2) d\mi(x)}{\int_{\R^{n-1}} \1_K(x,z_1) d\mi(x)} \geq \frac{\int_\bA \1_K(x,z_2) d\mi(x)}{\int_\bA \1_K(x,z_1) d\mi(x)}.\end{equation} 
Switching the left numerator with the right denominator we get the thesis.

If the right-hand side denominator in inequality (\ref{pndg1}) is zero, the right-hand side numerator is also zero, as $z_1 < z_2$ and $\KO$ is a c-set. Thus both for $z_1$ and $z_2$ our function is either zero or undefined.

If the left-hand side denominator is zero and the right-hand side is defined, again the left-hand side numerator is zero, thus in inequality (\ref{pndg2}) we have an equality, which again gives the thesis.
\end{proof}

\subsection{The general case --- the $\psi$ function}
Let $\lambda_K$ denote the Lebesgue measure restricted to $\KO$.
Recall that we set out to prove $$\lambda_K(\bA \times B ) \cdot \lambda_K(A \times \bB ) \geq \lambda_K(A \times B ) 
 \cdot \lambda_K(\bA \times \bB )$$ for any c-sets $A \subset \R^k$ and $B \subset \R^{n-k}$. This is equivalent to 
$$
\lambda_K(A \times \bB ) \cdot \lambda_K(\bA \times \R^{n-k}) \geq \lambda_K(\bA \times \bB) 
 \cdot \lambda_K(A \times \R^{n-k}).
$$ If either $\lambda_K(A \times \R^{n-k})$ or $\lambda_K(\bA \times \R^{n-k})$ is zero, then respectively either $\lambda_K(A \times \bB)$ or $\lambda_K(\bA \times \bB)$ is zero and the thesis is satisfied. Thus it suffices to prove
$$
\frac{\lambda_K(A \times \bB )}{\lambda_K(A \times \R^{n-k})} = \frac{\int_A \int_\bB \1_K(z,x) dz dx}{\int_A \int_{\R^{n-k}} \1_K(z,x) dz dx} \geq \frac{\int_\bA \int_\bB \1_K(z,x) dz dx}{\int_\bA \int_{\R^{n-k}} \1_K(z,x) dz dx} = \frac{\lambda_K(\bA \times \bB)}{\lambda_K(\bA \times \R^{n-k})},
$$
when both sides are defined, which means it is enough to prove $\psi_1(x) = \int_{\R^{n-k}} \1_K(z,x) dz$ and $\psi_2(x) = \int_\bB \1_K(z,x) dz$ define a non-degenerate $\Theta$ function on $K' = K_{z = 0} \subset \R^k$ and apply Theorem \ref{main}.

\begin{lemma} \label{dupaderivative} If $\bar{K}'$ is a derivative of $K'$, then there exists a generalized Orlicz ball $\bar{K}$ such that $\psi_1(x)$ on $\bar{K}'$ is equal to $\int_{\R^{n-k}} \1_{\bar{K}}(z,x) dz$ and $\psi_2(x)$ is equal to $\int_\bB \1_{\bar{K}}(z,x) dz$. \end{lemma}

The proof is identical to the proof of Lemma \ref{metaderivative}.

\begin{prop} For any generalized Orlicz ball $K \subset \R^n$, any coordinate-wise decomposition $\R^n = \R^k \times \R^{n-k}$ and any c-set $B \subset \R^{n-k}$ the functions $\psi_1$ and $\psi_2$ define a $\Theta$ function on $K$.\end{prop}

\begin{proof}
Property \ref{t-1} follows from the fact that $K$ is bounded. Property \ref{t0} follows from the fact $\KO$ is a c-set. Property \ref{t1} follows from the fact that $B \subset \R^{n-k}$. As before, the tricky part is to prove property \ref{t2}. Consider any coordinate-wise decomposition  $\R^k = \R^{k_1} \times \R^{k_2}$. Choose any variable $v$ in $\R^{k_1}$ and fix all the others at some fixed $\mathbf{v}_0$. We have:
$$\frac{\int_{\R^{k_2}} \psi_2(v,\mathbf{v}_0,y) d\mi(y)}{\int_{\R^{k_2}} \psi_1(v,\mathbf{v}_0,y) d\mi(y)}= \frac{\int_{\R^{k_2}} \int_\bB \1_K(v,\mathbf{v}_0,y,z) d\mi(y) dz}{\int_{\R^{k_2}} \int_{\R^{n-k}} \1_K(v,\mathbf{v}_0,y,z) d\mi(y) dz} = \frac{\int_{\R^{k_2} \times \bB} \1_K(v,\mathbf{v}_0,y,z) d\mi(y) dz}{\int_{\R^{l_2} \times \R^k} \1_K(v,\mathbf{v}_0,y,z) d\mi(y) dz}.$$

We have to prove this function is decreasing in $v$ where defined. Let us restrict ourselves to the generalized Orlicz ball $\hat{K} = K_{\mathbf{v} = \mathbf{v}_0}$. Notice that $\R^{l_2} \times A$ is a c-set in $\R^{l_2} \times \R^k$ and $\mi \otimes \lambda$ is a proper measure in $\R^{l_2} \times \R^k$. We have to prove 
$$\frac{\int_{\R^{l_2} \times \bA} \1_{\hat{K}}(v,y,z) d(\mi\otimes \lambda)(y,z)}{\int_{\R^{l_2} \times \R^k} \1_{\hat{K}}(v,y,z) d(\mi \otimes \lambda)(y,z)}$$
is decreasing in $v$, but this is exactly the thesis of Corollary \ref{wniosekjedno}.

Again, as in Lemma \ref{phitheta}, due to Lemma \ref{dupaderivative}, the appropriate ratio is also decreasing for any derivative $\bar{K}$ of $K$.
\end{proof}

\begin{prop} 
For any generalized Orlicz ball $K \subset \R^n$, any coordinate-wise decomposition $\R^n = \R^k \times \R^{n-k}$ and any c-set $B \subset \R^{n-k}$ the functions $\psi_1$ and $\psi_2$ define a non-degenerate $\Theta$ function on $K$.\end{prop}
\begin{proof}
Again the derivatives of $\psi$ are again functions formed as in Lemma \ref{dupaderivative}, so it is enough to prove $\psi$ is weakly non-degenerate.

Take any $\eps > 0$. From Lemma \ref{approrlicz} we may take a proper generalized Orlicz ball $\hat{K} \subset K$ with $\lambda(K \setminus \hat{K}) < \eps \min\{\lambda(K),1\} \slash 2$ and $\lambda_k(K_{z=0} \setminus \hat{K}_{z=0}) < \eps \min\{\lambda_k(K_{z=0}),1\})$ from Lemma \ref{apprsection}. Denote $\hat{K}_{z=0}$ by $\hat{K}'$. 

Let $z_1$ be any coordinate in $\R^{n-k}$, take $B' = B \cup (\{\mathbf{z}: z_1 < \delta\} \cap \KO)$, where $\delta$ is so small that the addition is of $\lambda_{n-k}$ measure less than $\eps \slash 2$. $B'$ is a sum of two c-sets and thus a c-set.

We define $\psi_1'(x) = \int_{\R^{n-k}} \1_{\hat{K}}(z,x) dz$ and $\psi_2'(x) = \int_{\bB'} \1_{\hat{K}}(z,x) dz$. We have $\lambda_k(K' \setminus \hat{K}') < \eps \lambda_k(K')$ from the definition of $\hat{K}$. Also $\psi_1'$ and $\psi_2'$ are indeed good approximations of $\psi_1$ and $\psi_2$, as 
$$\int_{\R^k} |\psi_1(x) - \psi_1'(x)| dx \leq \int_{\R^k} \int_{\R^{n-k}} |\1_K(x,z) - \1_{\hat{K}}(x,z)| dz dx = \lambda (K \bigtriangleup \hat{K}) \leq \eps \slash 2,$$
and
\begin{align*} \int_{\R^k} |\psi_2(x) - \psi_2'(x)| dx &= \int_{\R^k} \Big|\int_{\R^{n-k}} \1_K(x,z) \1_\bB(x,z) - \1_{\hat{K}}(x,z) \1_{\bB'}(x,z) dz\Big|dx \leq \mi((K \cap \bB) \bigtriangleup (\hat{K} \cap \bB')) \\ &\leq \lambda(K \setminus \hat{K}) + \lambda_K (\bB \bigtriangleup \bB') = \lambda(K \setminus \hat{K}) + \lambda_K (B \bigtriangleup B') \leq \eps.\end{align*} Thus we only have to prove that $\psi_1'$ and $\psi_2'$ define a strict $\Theta$ function on $\hat{K}$.

Property (\ref{s3}) is true as $\hat{K}$ is proper --- $\hat{K}'$ is defined by those Young functions of $\hat{K}$ which act on the variables of $\R^k$. Property (\ref{t4}) is obvious from the definition of $\psi_1'$. The function $\psi_2'$ is 0 on the set $\sum f_i(x_i) > 1 - f_{z_1}(\delta)$ from the definition of $B'$ --- any point in $\bB'$ has $z_1 > \delta$, hence property (\ref{s2}). Finally (\ref{s4}) is checked exactly as in Lemma \ref{phipthetas}.
\end{proof}

Thus $\psi_1$ and $\psi_2$ do define a non-degenerate $\Theta$ function, which ends the proof of Theorem \ref{orliczglownetw}.

\end{document}